\documentclass[12pt, reqno]{amsart}
\usepackage{textcomp}
\usepackage{amsmath,amsfonts,amssymb,amsthm, multicol}
\usepackage{anysize}
\marginsize{2cm}{2cm}{2cm}{2cm}
\usepackage{graphicx}
\usepackage{color}
\usepackage[pdftex]{hyperref}
\usepackage{comment}
\usepackage{tikz}
    \usetikzlibrary{cd} 
    \usetikzlibrary{calc} 

\usepackage{enumerate}
\usepackage{caption}
\usepackage{ gensymb }

\usepackage{mathtools}

\def\a{\mathfrak{a}}
\def\b{\mathfrak{b}}
\def\g{\mathfrak{g}}
\def\h{\mathfrak{h}}
\def\k{\mathfrak{k}}

\def\n{\mathfrak{n}}
\def\q{\mathfrak{q}}
\def\r{\mathfrak{r}}

\def\t{\mathfrak{t}}

\def\z{\mathfrak{z}}

\def\sl{\mathfrak{sl}}
\def\su{\mathfrak{su}}
\def\sp{\mathfrak{sp}}

\def\C{\mathbb{C}}
\def\R{\mathbb{R}}

\def\N{\mathbb{N}}

\def\ad{\operatorname{ad}}
\def\tr{\operatorname{tr}}

\def\alt{\raise1pt\hbox{$\bigwedge$}}

\def\im{\operatorname{im}}

\def\Id{\operatorname{Id}}

\def\End{\operatorname{End}}
\def\Hom{\operatorname{Hom}}

\def\mid{\, \vert \,} 
\def\Span{\operatorname{span}}
\def\spec{\operatorname{spec}}
\def\diag{\operatorname{diag}}
  
\newcommand\aff{\mathfrak{aff}}

\theoremstyle{plain}
\newtheorem{theorem}{\bf Theorem}[section]
\newtheorem{corollary}[theorem]{\bf Corollary}
\newtheorem{proposition}[theorem]{\bf Proposition}
\newtheorem{lemma}[theorem]{\bf Lemma}

\theoremstyle{definition}

\newtheorem{example}[theorem]{\bf Example}

\theoremstyle{remark}
\newtheorem{remark}[theorem]{\bf Remark}

\usepackage[pageref]{backref}
\renewcommand*{\backrefalt}[4]{%
  \ifcase #1 %
    (Not cited.)%
  \or
    (Cited on page~\hyperlink{page.#2}{#2}.)%
  \else
    (Cited on pages~\hyperlink{page.#2}{#2}.)%
  \fi
}

\title{On centerless unimodular contact Lie algebras.}

\author{Agustín Garrone}
\email{agustin.garrone@unc.edu.ar} 

\date{}
\address{FAMAF, Universidad Nacional de C\'ordoba and CIEM-CONICET, Av. Medina Allende s/n, Ciudad Universitaria, X5000HUA C\'ordoba, Argentina}

\thanks{This work was partially supported by CONICET and SECyT-UNC (Argentina)}

\subjclass[2020]{}

\begin{document}

\begin{abstract}
    We provide an elementary proof that, in a (transversely) unimodular contact Lie algebra, the adjoint action of the Reeb vector is nilpotent except when the Lie algebra is isomorphic to either $\sl(2,\R)$ or $\su(2)$. We introduce a class of contact Lie algebras, called \textit{DS-contact Lie algebras}, containing all K-contact Lie algebras, and deduce from the previous result that the only centerless unimodular examples in this class are precisely $\sl(2,\R)$ and $\su(2)$. This gives an alternative proof of the previously known fact that centerless unimodular Sasakian Lie algebras are isomorphic to either $\sl(2,\R)$ or $\su(2)$. Some other results known to hold for Sasakian Lie algebras are generalized as well. We investigate several properties of DS-contact Lie algebras in relation to Frobenius Lie algebras, and also classify them in dimension five. Some implications for the contact Lefschetz condition are explored. 
\end{abstract}

\maketitle

\section{Introduction}

    A contact form on a $(2n+1)$-dimensional Lie algebra $\g$ is a $1$-form $\eta \in \g^*$ such that $\eta \wedge (d_{\g}\eta)^n \neq 0$. Such a form determines a unique element $\xi \in \g$, called the \textit{Reeb vector}, characterized by $\eta(\xi)=1$ and $\iota_{\xi} d \eta =0$. It is known that the center of a contact Lie algebra is either trivial or generated by the Reeb vector (see Proposition~\ref{prop: center in contact Lie algebras}). Moreover, contact Lie algebras with nontrivial center are precisely the one-dimensional central extensions of symplectic Lie algebras (see Proposition~\ref{prop: contactization}). Since symplectic Lie algebras are somewhat better understood, this correspondence provides useful structural information about contact Lie algebras with nontrivial center, and much of what is known comes from this fact

    \indent The situation is different for centerless contact Lie algebras, since no analogous reduction to symplectic Lie algebras is available. Several approaches have been used, including deformations and contractions of Lie algebras~\cite{GR}, invariants~\cite{S}, and double extensions~\cite{ARVS 2, ARVS 3, RVSSV}. Another common strategy is to focus on special classes of contact Lie algebras, such as those with abelian nilradical~\cite{ARVS 1}, or with a codimension-one Frobenius ideal~\cite{D}, or the so-called \textit{seaweed} Lie algebras~\cite{CR,CMRS,MY}. Considerable attention has also been devoted to compatible geometric structures, especially Sasakian structures and related variants~\cite{AFV,ACN,CF,DPS,LL}. Several of these techniques also work for the nontrivial center case. It should be noted that, by results of Gromov \cite{G} and Borman-Eliashberg-Murphy \cite{BEM}, every odd-dimensional Lie group admits a possibly noninvariant contact structure. Notice that not every odd-dimensional Lie algebra admits a contact form (e.g., the abelian ones).

    \indent Despite these efforts, relatively little is known about the general structure of centerless contact Lie algebras. One notable exception is the the fact that the only centerless unimodular Sasakian Lie algebras are $\sl(2,\R)$ and $\su(2)$. This arises as an immediate consequence of \cite[Theorem 2.1]{AHK}, where a classification of unimodular Sasakian Lie algebras up to modification is established: besides $\sl(2,\R)$ and $\su(2)$, they are Heisenberg Lie algebras $\h_{2n+1}$. To achieve this classification, fundamental results on the structure of unimodular Kähler Lie algebras of Hano \cite{Hano} and Nakajima \cite{Nakajima} are used, as well as properties of the aforementioned modification technique.   

    \indent While not a full-blown classification result, a much more general result can be established by rather elementary means. Its proof is the main content of this article:      
\begin{theorem} \label{thm: main result}
    With the exception of $\sl(2,\R)$ and $\su(2)$, the map $\ad_{\xi}$ in a (transversely) unimodular contact Lie algebra $(\g,\eta)$ is nilpotent, where $\xi$ is the Reeb vector.
\end{theorem}
    \indent Recall that a Lie algebra $\g$ is said to be unimodular if $\tr(\ad_x) = 0$ for all $x \in \g$. There is a well-known equivalence between unimodularity and the validity of Poincaré duality for the Chevalley-Eilenberg cohomology of $\g$. If $\g$ admits a contact form $\eta$, then a similar equivalence for the $\xi$-basic cohomology can be defined, where $\xi$ is the Reeb vector of $(\g,\eta)$: the relevant notion is that of \textit{transverse unimodularity}, meaning that $\tr(M_x) = 0$ for all $x \in \ker(\ad_{\xi} \vert_{\h})$; here, $M_x := \mathrm{pr}_{\h} \circ \ad_x \vert_{\h}$, where $\h := \ker \eta$ and $\mathrm{pr}_{\h}: \g \to \h$ is the linear projection associated with the vector space splitting $\g = \R \xi \oplus \h$. Of course, unimodularity implies transverse unimodularity, but not the other way around. We refer to Subsection \ref{section: cohomology remarks}  for details. 

    \indent Unimodular contact Lie algebras with nilpotent $\ad_{\xi}$ do exist, as seen in Example \ref{ex: Diatta-Foreman}. The statement that contact Lie algebras with nilpotent $\ad_{\xi}$ are unimodular is unsurprisingly false, as seen in Examples \ref{ex: converse of main is false 1} and \ref{ex: converse of main is false 2}. Since K-contact Lie algebras have skew-symmetric $\ad_{\xi}$, we obtain an immediate consequence:  
\begin{corollary} \label{cor: main corollary}
    Centerless unimodular K-contact Lie algebras are either $\sl(2,\R)$ or $\su(2)$. 
\end{corollary}
    \indent The same is true for centerless contact Lie algebras $(\g,\eta)$ admitting a vector space splitting $\g = \ker \ad_{\xi} \oplus \im \ad_{\xi}$, which we refer to as \textit{DS-contact Lie algebras}. The ``DS" part in the name stands for \textit{direct sum}. Notice that not all contact Lie algebras are of this shape, for the splitting is equivalent to the condition $\ker \ad_{\xi}^2 = \ker \ad_{\xi}$: it is thus a constraint on the Jordan form of $\ad_{\xi}$. Section \ref{section: DS-contact} is devoted to the study of this class of contact Lie algebras.  

    \indent For a DS-contact Lie algebra $(\g,\eta)$, set $\t := \ker \ad_{\xi}$, $\q := \im \ad_{\xi}$, and $\t_0 := \t \cap \h$. A straightforward verification shows that the restrictions $\omega_{\t} := - (d_{\g} \eta) \vert_{\t}$ and $\omega_{\q} := - (d_{\g} \eta) \vert_{\q}$ are nondegenerate on $\t_0$ and $\q$, respectively (see Lemma \ref{lemma: simplecticidad}). Moreover, it can be checked just as easily that $\t_0$ has a natural Lie bracket structure for which $\omega_{\t}$ is closed, and that $\t$ is a 1-dimensional central extension of $\t_0$ with associated cocycle $\omega_{\t}$ (we refer to Subsection \ref{section: elementary properties} for details). As an interesting byproduct of the proof of Theorem \ref{thm: main result}, we obtain the~next~result:  
\begin{theorem} \label{thm: t0 is frobenius}
    $\omega_{\t}$ is an exact symplectic form on $\t_0$. Thus, neither $\t_0$ nor $\t$ is unimodular.  
\end{theorem}
    \indent Lie algebras admitting exact symplectic forms are known as \textit{Frobenius Lie algebras} in the literature. An explicit primitive of $\omega_{\t}$ can be given (see Theorem \ref{thm: t_0 is frobenius}). The Frobenius Lie algebra $(\t_0, \omega_{\t})$ is somewhat restricted (see Corollary \ref{cor: some obstructions}). A natural question, which we call the \textit{realization problem}, can be posed: which Frobenius Lie algebras can be realized as the $\t_0$-part of some DS-contact Lie algebra? For a given Frobenius Lie algebra $(\a, \omega_{\a})$, set
\begin{gather*} 
    \mathcal{R}(\a, \omega_{\a}) := \{d \in 2 \mathbb N \mid (\a, \omega_{\a}) \text{ is realizable with $\dim \q = d$} \}, \\ \nu(\a) := \min\{\dim J \mid \text{$J$ is a nonzero ideal in $\a$}\}.
\end{gather*}
    \noindent The next result summarizes what has been proven in relation to the realization problem in this article (see Proposition \ref{prop: realizable 1}, Proposition \ref{prop: realizable 2}, and Corollary \ref{cor: obstruction}): 
\begin{proposition} \label{prop: the realization problem}
    \phantom{.}
\begin{enumerate} [\rm (i)]
    \item $2\in \mathcal{R}(\a, \omega_{\a})$ for any Frobenius Lie algebra $(\a, \omega_{\a})$. 
    \item $\mathcal{R}(\a, \omega_{\a}) = 2\mathbb N$ for any Frobenius Lie algebra $(\a, \omega_{\a})$ with a one-dimensional ideal. 
    \item If $k^2 < \nu(\a)$ for some $k \geq 2$ then $2k \notin \mathcal{R}(\a, \omega_{\a})$. 
\end{enumerate} 
\end{proposition}
    \indent Another interesting byproduct of the proof of Theorem \ref{thm: main result} is the following result, which also generalizes \cite[Proposition 8.1]{ACN}. Here, $\g$ is \textit{not} assumed to be DS-contact:  
\begin{proposition} \label{prop: main proposition}
    Let $(\g,\eta)$ be a contact Lie algebra for which $\ker (K_s) = \R \xi$, where $K_s$ denotes the semisimple part of $\ad_{\xi}$. Then $\g$ is isomorphic to either $\su(2)$ or $\sl(2,\R)$. Thus, $\g$ is three-dimensional and simple, and admits compatible Sasakian structures. 
\end{proposition} 
    \indent An important application of Theorem \ref{thm: main result} is cohomological in nature, and concerns the study of Lefschetz-type conditions for contact Lie algebras. Recall that the usual Lefschetz condition is defined for symplectic Lie algebras: a $2n$-dimensional symplectic Lie algebra $(\h, \omega)$ is said to be $s$-Lefschetz, where $0 \leq s \leq n-1$, if the maps
\begin{align*}
    L^{n-k}:H^k(\h) \to H^{2n-k}(\h), \quad L^{n-k}([\alpha]) := [\omega^{n-k} \wedge \alpha]
\end{align*}
    \noindent are bijective for all $k \leq s$. There are two main contact analogues for compact contact manifolds: one using basic cohomology, as put forward by \cite{KA}, and another using de Rham cohomology, as introduced in \cite{Cagliari}. Both can be considered at the Lie algebra level: the former for basic cohomology, the latter for Chevalley-Eilenberg cohomology. Remarkably, in both \cite{KA} and \cite{Cagliari} it is shown that compact Sasakian manifolds are $(n-1)$-Lefschetz, or \textit{hard-Lefschetz}, in each corresponding sense. It is proven in \cite[Theorem 6.3]{Lin} that, for a compact $(2n+1)$-dimensional K-contact manifold and $0 \leq s \leq n-1$, the basic $s$-Lefschetz condition holds if and only if the de Rham $s$-Lefschetz condition holds. A similar equivalence can be formulated at the Lie algebra level.

    \indent Corollary \ref{cor: main corollary} ensures that, in the centerless K-contact setting, the two versions of the Lefschetz condition fail simultaneously already at the $s = 0$ level (except for $\sl(2,\R)$ and $\su(2)$, which are Sasakian and hence hard-Lefschetz). More generally, Theorem \ref{thm: main result} ensures that the same simultaneous failure holds for any contact Lie algebra for which $\ad_{\xi}$ is not nilpotent, even though Lin's equivalence is not known to hold in such generality. The Lefschetz questions considered here, both in the symplectic and contact settings, are connected with previous work of the author \cite{AG1, AG2, AG3}. The status of Lin's equivalence for general contact Lie algebras remains open; in particular, to the best of the author's knowledge, there has been no systematic study of any version of the contact Lefschetz condition when $\ad_{\xi}$ is nilpotent. We intend to address this in future work.  

    \indent The article is organized as follows: Section \ref{section: preliminaries} contains necessary preliminaries to keep the discussion reasonably self-contained; while the information in Subsections \ref{section: contactization} and \ref{section: sasaki and K-contact} is well-known, Subsection \ref{section: cohomology remarks} contains somewhat novel results. In Section \ref{section: a useful decomposition} a general vector space decomposition of contact Lie algebras is discussed, as well as many consequences stemming from the interaction between the Lie bracket and said decomposition, all serving to establish Theorem \ref{thm: main result} and the subsequent study of DS-contact Lie algebras; also, the proof of Proposition \ref{prop: main proposition} is given. The full proof of Theorem \ref{thm: main result} is contained in Section \ref{section: proof of the main result}. The notion of DS-contact Lie algebras is introduced in Section \ref{section: DS-contact}, and studied as indicated above: Subsection \ref{section: elementary properties} lists elementary properties of DS-contact Lie algebras, Subsection \ref{section: on frobenius lie algebras} contains the proof of Theorem \ref{thm: t0 is frobenius} and Proposition \ref{prop: the realization problem}, Subsection \ref{section: a classification result in dimension five} contains a classification of 5-dimensional (centerless) DS-contact Lie algebras, and Subsection \ref{section: more cohomology remarks} contains a refinement of some ideas presented in Subsection \ref{section: cohomology remarks}. 

\textit{Acknowledgements.} The author is very grateful to Dr. Adrián Andrada for his continued guidance during the preparation of this article, for suggesting the question that led to the main result, for many valuable comments, and for carefully reading the manuscript. The author also wishes to thank Dr. Ignacio Bono Parisi and Facundo Javier Gelatti for their unwavering support and encouragement.

\section{Preliminaries on contact Lie algebras} \label{section: preliminaries}

\indent Let $\g$ be a finite-dimensional real Lie algebra. A \textit{contact form} on $\g$ is a $1$-form $\eta \in \alt^1 \g^*$ with
\begin{align*}
    \eta \wedge (d \eta)^n \neq 0, \quad \text{where $\dim \g = 2n +1$}.
\end{align*}
    \noindent The pair $(\g, \eta)$ is called a \textit{contact Lie algebra}. In every contact Lie algebra $(\g,\eta)$, the conditions
\begin{align*}
    \iota_{\xi} \eta = 1 \quad \text{ and } \quad \iota_{\xi} (d_{\g} \eta) = 0
\end{align*}
    \noindent uniquely determine a vector $\xi \in \g$, called the \textit{Reeb vector} of $(\g,\eta)$. 
\begin{lemma}
    The following conditions hold.   
\begin{multicols}{2}
\begin{enumerate} [\rm (i)]
    \item $\ad_{\xi}( \ker \eta) \subseteq \ker \eta$.
    \item $\im(\ad_{\xi}) \subseteq \ker \eta$.
\end{enumerate}
\end{multicols} 
\end{lemma}
\begin{proof}
    Notice that 
\begin{align} \label{eq: Seta = 0}
    0 = \iota_{\xi}(d_{\g} \eta) = (d_{\g} \circ \iota_{\xi} + \iota_{\xi} \circ d_{\g})(\eta) = \mathcal{L}_{\xi} \eta = \ad_{\xi}^* \eta. 
\end{align}
\begin{enumerate} [\rm (i)]
    \item If $x \in \ker \eta$ then $\eta([\xi,x])=\eta(\ad_{\xi} x) = \ad_{\xi}^*\eta(x) = 0$. Thus $\ad_{\xi}( \ker \eta) \subseteq \ker \eta$.
    \item If $y \in \im \ad_{\xi}$ then $y = \ad_{\xi} x$ for some $x \in \g$, and then $\eta(y) = \eta( \ad_{\xi} x )= \ad_{\xi}^* \eta(x) = 0$. Therefore, $\mathrm{im}(\ad_{\xi}) \subseteq \ker \eta$. \qedhere 
\end{enumerate} 
\end{proof}
    \indent Thus, $\h := \ker \eta$ is an $\ad_{\xi}$-invariant subspace of $\g$. Also, the contact condition on $\eta$ implies that $\omega := - (d_{\g} \eta)$ is a nondegenerate $2$-form on $\h$. Let $p_{\h}\colon\g \to \h$ denote the canonical linear projection associated with the vector space splitting $\g = \R \xi \oplus \h$, and denote
\begin{align*}
    \alt_{\xi}^* \g^* := \{ \alpha \in \alt^* \g^* \mid \iota_{\xi} \alpha = 0\}. 
\end{align*}
    \indent The following result follows from direct, unenlightnening computations. 
\begin{lemma} \label{lemma: an identification}
    The induced map $p_{\h}^*:\alt^* \h^* \to \alt^* \g^*$ is a linear isomorphism onto $\alt_{\xi}^* \g^*$. Its inverse is the map $\mathrm{res}:\alt_{\xi}^* \g^* \to \alt^* \h^*$ given by restriction $\mathrm{res}(\alpha) := \alpha \vert_{\h}$. 
\end{lemma}
    \indent We identify $\alt_{\xi}^* \g^* \cong \alt^* \h^*$ in what follows. Notice that this is \textit{not} an isomorphism of cochain complexes, since in general there is no natural differential on $\alt^* \h^*$. 

\subsection{Contactization} \label{section: contactization}

    \indent The class of contact Lie algebras with nontrivial center is structurally well understood.  
\begin{proposition} \cite[Proposition 1]{AFV} \label{prop: center in contact Lie algebras}
    If the center $\z(\g)$ of $(\g,\eta)$ is nontrivial then $\z(\g) = \R\xi$. 
\end{proposition}
\begin{proof}
    Let $z \in \z(\g)$ be written as $z = a \xi + x$ for some $a \in \R$ and $x \in \ker \eta$. Since  $\iota_{\xi} (d_{\g} \eta) = 0$ (due to the definition of $\xi$) and $\iota_{z} (d_{\g} \eta) = 0$ (since $z \in \z(\g)$), we get
\begin{align*}
    0 = \iota_z (d_{\g}\eta) = a \iota_{\xi} (d_{\g}\eta) + \iota_x (d_{\g}\eta) = \iota_x (d_{\g}\eta).
\end{align*}
    \noindent The nondegeneracy of $d_{\g} \eta$ on $\ker \eta$ implies $x = 0$, and thus $z = a \xi$. Therefore, either $\z(\g) = 0$ (when $a = 0$) or $\z(\g) = \R \xi$ (when $a \neq 0$). 
\end{proof} 
    \indent The following result is well-known, and it established by routine computations. 
\begin{proposition} \cite[Proposition 2]{AFV} \label{prop: contactization}
    If $\z(\g) = \R \xi$ then $\h$ is a Lie algebra with bracket $[\cdot, \cdot]_{\h}$ uniquely determined by declaring $p_{\h}:\g \to \h$ to be a Lie algebra morphism. In this case, $\g$ is a $1$-dimensional central extension of $\h$ by $\omega$, meaning 
\begin{align} \label{eq: contactization bracket}
    [x,y]_{\g} := \omega(x,y) \xi + [x,y]_{\h} \text{ for all $x$, $y \in \h$}, \quad [\xi, \h] = 0. 
\end{align}
\end{proposition}
    \indent A $1$-dimensional central extension by a symplectic cocycle is sometimes called \textit{contactization}. In fact, contact Lie algebras with nontrivial center are in bijective correspondence with symplectic Lie algebras, contactization being the correspondence (see \cite[Proposition 2]{AFV}): Starting with a symplectic Lie algebra $(\h,\omega)$, equation \eqref{eq: contactization bracket} defines a Lie bracket on the vector space $\g := \R \xi \oplus \h$, and the formula  $\eta(a \xi + x) = a$ for all $a \in \R$ and $x \in \h$ defines a contact form on $\g$ with Reeb vector $\xi$; certainly, this construction gives $\z(\g) = \R \xi$, so contactization occurs if and only if $\g$ has nontrivial center. There is no similar result characterizing all centerless contact Lie algebras. Certainly, two symplectic Lie algebras are isomorphic if and only if their contactizations are isomorphic (see \cite[Proposition 4]{AFV}). 

\begin{remark} \label{obs: when are dg and dh equal}
    If $p_{\h}:\g \to \h$ is a Lie algebra morphism then the transpose map $p_{\h}^*:\h^* \to \g^*$ induces a cochain map between $(\alt^* \h^*, d_{\h})$ and $(\alt^* \g^*, d_{\g})$, and a cochain isomorphism between $(\alt^* \h^*, d_{\h})$ and $(\alt_{\xi}^* \g^*, d_{\g})$. Following the remarks in the last paragraph, the maps $p_{\h}$ and $p_{\h}^*$ can be turned into a Lie algebra morphism and a cochain map respectively precisely when $(\g, \omega)$ arises as contactization from a symplectic Lie algebra $(\h,\omega)$. 
\end{remark}    
\begin{remark} \label{obs: iff unimod, nil, solv}
    If $(\g,\eta)$ is a contact Lie algebra obtained from a symplectic Lie algebra $(\h,\omega)$ via contactization then, with respect to the vector space splitting $\g = \R \xi \oplus \h$, 
\begin{align*}
    \ad_x^{\g} = 
    \begin{pmatrix}
        0 & * \\ 
        0 & \ad_x^{\h}
    \end{pmatrix} \quad \text{for all $x \in \h \subseteq \g$}.
\end{align*}
    \noindent In particular, $\g$ is unimodular, nilpotent, (completely) solvable if and only if $\h$ is unimodular, nilpotent, (completely) solvable: notice that Engel's theorem is used to establish the equivalence for nilpotency, and Cartan's criterion is used to establish the equivalence for solvability. In the unimodular case, it follows from \cite[Theorem 11]{Chu} that $\h$ is solvable since it is symplectic. 
\end{remark}  

\subsection{Sasakian and K-contact Lie algebras} \label{section: sasaki and K-contact}

\indent On any contact Lie algebra $(\g,\eta)$, there exist an inner product $g$ and a $(1,1)$-tensor field $\Phi$ subject to the following identities:
\begin{align*}
    \eta = \iota_{\xi} g, \quad d \eta = 2 g(\cdot, \Phi \, \cdot), \quad \Phi^2 = - \Id_{\g} + \eta \otimes \xi.
\end{align*}
    \noindent All these imply at once that $\Phi \xi = 0$ and $\eta \circ \Phi = 0$, as well as the compatibility condition
\begin{align*}
    g(\Phi x, \Phi y) = g(x,y) - \eta(x) \eta(y) \text{ for all $x$, $y \in \g)$}. 
\end{align*}
    \noindent One can adapt the arguments in \cite[Section 4]{Blair} for the Lie algebra setting to establish these results. A triple $(\eta, g, \Phi)$ satisfying all of the above is called a \textit{metric contact structure}. Special metric contact structures include \textit{K-contact} and \textit{Sasakian} structures, which we now define.

    \indent Denote by $\mathcal{L}_{\xi}$ the Lie derivative operator $\mathcal{L}_{\xi}$ on $\g$ with respect to $\xi$, by $\nabla$ the Levi-Civita connection associated with the inner product $g$ on $\g$, and by $N_{\Phi}$ the Nijenhuis tensor associated to $\Phi$. Recall that it satisfies, in the metric contact setting, the following identities: 
\begin{align*}
    \mathcal{L}_{\xi} \Phi = [\ad_{\xi}, \Phi], \quad (\mathcal{L}_{\xi} g)(x,y) = g(\ad_{\xi} x,y) + g(x, \ad_{\xi} y) \text{ for all $x$, $y \in \g$}, 
\end{align*}

    \indent A K-contact Lie algebra is a Lie algebra $\g$ with a metric contact structure $(\eta, g, \Phi)$ in which any, and therefore all, of the following equivalent statements hold: 
\begin{proposition} \label{prop: K-contact equivalences}
    The following conditions are equivalent:
\begin{multicols}{2}
\begin{enumerate} [\rm (i)]
    \item $\mathcal{L}_{\xi} g = 0$. 
    \item $\mathcal{L}_{\xi} \Phi = 0$. 
    \item $\ad_{\xi}$ is skew symmetric with respect to $g$.
    \item $\ad_{\xi}$ and $\Phi$ commute. 
    \item $\ad_{\xi} \circ \Phi$ is symmetric with respect to $g$.
    \item $\Phi x = - \nabla_x \xi$ for all $x \in \g$. 
\end{enumerate}    
\end{multicols}
\end{proposition}
\begin{proof}
    It follows by adapting \cite[Theorem 6.2]{Blair} and \cite[Lemma 6.2]{Blair} to the Lie algebra setting.  
\end{proof}
    \indent As a direct consequence, in any K-contact Lie algebra there is a $g$-orthogonal direct sum decomposition of $\g$ into $\Phi$-invariant subspaces:
\begin{align} \label{eq: DS-contact equation}
    \g = \ker \ad_{\xi} \oplus \im \ad_{\xi}. 
\end{align} 
    \indent A Sasakian Lie algebra is a Lie algebra $\g$ with a metric contact structure $(\eta, g, \Phi)$ in which any, and therefore all, of the following equivalent statements hold: 
\begin{proposition} \label{prop: Sasaki equivalences}
    The following conditions are equivalent:
\begin{multicols}{2}
\begin{enumerate} [\rm (i)]
    \item $N_{\Phi}(x,y) = - d \eta(x,y) \xi$ for $x$, $y \in \g$.  
    \item $(\nabla_x \Phi)y = g(x,y) \xi - \eta(y) x$ for $x$, $y \in \g$.   
\end{enumerate}    
\end{multicols}
\end{proposition}
\begin{proof} 
    It follows by adapting \cite[Theorem 6.3]{Blair} to the Lie algebra setting. 
\end{proof}
    \indent Recall that all Sasakian Lie algebras are K-contact (as pointed out in \cite[Corollary 6.3]{Blair}), but the converse does not hold in general: for instance, it is known that the only Sasakian nilpotent Lie algebras are the Heisenberg Lie algebras (see \cite[Theorem 3.9]{AFV}), while all nilpotent Lie algebras are K-contact (since they have nontrivial center and thus $\ad_{\xi} = 0$ is skew-symmetric with respect to any compatible metric). 

\subsection{Cohomology remarks} \label{section: cohomology remarks}
        
    \indent It is well-known that a $m$-dimensional Lie algebra $\k$ is unimodular if and only if $H^m_{CE}(\k) \neq 0$, meaning that the differential map $d_{\k}: \alt^{m-1} \k^* \to \alt^m \k^*$ is zero. Here, $H^*_{CE}(\g)$ stands for Chevalley-Eilenberg cohomology with trivial coefficients. This follows from the fact that the map $x \mapsto \iota_x \mu$ is an isomorphism between $\k$ and $\alt^{m-1} \k^*$, as well as the identity
\begin{align*}
    d_{\k}(\iota_x \mu) = \mathcal{L}_x \mu = - \tr(\ad_x) \mu \quad \text{for all $x \in \k$ and $\mu \in \alt^m \k^*$}. 
\end{align*}
    \noindent We refer to \cite[Section 3.1]{AG3} for the fleshed-out details. The next result appears to be new. 
\begin{proposition} \label{prop: its converse need not hold}
    If a contact Lie algebra $(\g, \eta)$ is unimodular then $\xi \in [\g,\g]$. 
\end{proposition}
\begin{proof}
    Assume $\xi \notin [\g,\g]$. Then there exists $\alpha \in \g^*$ such that $\alpha(\xi) = 1$ and $\alpha([\g,\g]) = 0$. In particular, $d_{\g} \alpha = 0$. Notice that $\alpha \wedge \omega^n$ is exact, for $d_{\g} \omega = 0$ and then 
\begin{align*}
    d_{\g} (\alpha \wedge \eta \wedge \omega^{n-1}) = - \alpha \wedge d_{\g} \eta \wedge \omega^{n-1} = \alpha \wedge \omega^n.
\end{align*}
    \noindent There is $\alpha_0 \in \g^*$ such that $\iota_{\xi} \alpha_0 = 0$ and $\alpha = \eta + \alpha_0$. Since $\iota_{\xi} \omega^n = 0$ as well, we see that $\alpha_0 \wedge \omega^n$ is identified with a form of degree $(2n+1)$ in $\h = \ker \eta$, and thus it is the zero form. This implies $\alpha \wedge \omega^n = \eta \wedge \omega^n$, whence the volume form $\eta \wedge \omega^n$ is exact and $H^{2n+1}(\g) = 0$. 
\end{proof} 
    \indent From Proposition \ref{prop: its converse need not hold} we get a particular case of Theorem \ref{thm: main result} for unimodular solvable Lie algebras: since $\xi \in [\g,\g]$ holds for this class of Lie algebras, where $[\g,\g]$ is a nilpotent ideal, Engel's theorem implies that $\ad_{\xi}$ is nilpotent.

    \indent The converse of Proposition \ref{prop: its converse need not hold} need not hold. The following set of examples exhibit diverse ways in which it fails. 

    \indent The first two examples arise as contactization of symplectic Lie algebras, and thus have nontrivial center and $\ad_{\xi} = 0$. One is solvable while the other is not. According to \cite[Corollary 5]{D}, there are no 5-dimensional nonsolvable examples of this kind. It is remarked in \cite[Remark 3.11]{AG3} that, for this class of contact Lie algebras, $\xi \in [\g,\g]$ holds precisely when $\omega := - d_{\g} \eta$ is nonexact on $\h$. 
   
\begin{example}
    Let $\h$ be the Lie algebra with basis $\{ e_1, e_2, e_3, e_4\}$ and brackets given by
\begin{align*}
    [e_1, e_4] = - e_1, \quad [e_3, e_4] = - e_2.
\end{align*}    
    \noindent It is not unimodular since $\tr(\ad_{e_4}) = 1$, and it is solvable; moreover, and $\omega := e^{14} + e^{23}$ is a (nonexact) symplectic form on $\h$. As known from Remark \ref{obs: iff unimod, nil, solv}, the contactization $(\g,\eta)$ of $(\h,\omega)$ is not unimodular due to Remark \ref{obs: iff unimod, nil, solv}. However, $\xi = \omega(e_2,e_3) \xi + [e_2, e_3]_{\h} = [e_2, e_3]_{\g} \in [\g,\g]$.  
\end{example} 
\begin{example}
    Let $\h$ be the Lie algebra with basis $\{e_1, e_2, e_3, e_4, e_5, e_6\}$ and bracket given by
\begin{align*}
    [e_2, e_5] = e_4, \quad [e_3, e_4] = e_5, \quad [e_4, e_6] = e_4, \quad [e_5, e_6] = e_5.
\end{align*}     
    \noindent It is not unimodular since $\tr(\ad_{e_6}) = -2$, and it is not solvable; moreover, $\omega := e^{12} + e^{15} - e^{34} - e^{56}$ is a (nonexact) symplectic form on $\h$. As known from Remark \ref{obs: iff unimod, nil, solv}, the contactization $(\g,\eta)$ of $(\h,\omega)$ is not unimodular due to Remark \ref{obs: iff unimod, nil, solv}. However, $\xi = \omega(e_1,e_2) \xi + [e_1,e_2]_{\h} = [e_1,e_2]_{\g} \in [\g,\g]_{\g}$. 
\end{example}
    \indent The next two examples have nilpotent $\ad_{\xi}$. Both are centerless, so $\ad_{\xi} \neq 0$; also, they are decomposable contact Lie algebras, arising as a direct sum of a contact Lie algebra and an exact symplectic Lie algebra. One is solvable while the other is not. They also show that the converse of Theorem \ref{thm: main result} is false. As deduced from \cite[Proof of Corollary 5]{D}, there are no 5-dimensional nonsolvable examples of this kind.  
\begin{example} \label{ex: converse of main is false 1}
    Let $\g$ be the Lie algebra with basis $\{\xi, e_1, e_2, e_3, e_4\}$ and bracket given~by
\begin{align*}
    [\xi, e_1] = e_2 - \xi, \quad [e_1, e_2] = e_2, \quad  [e_3, e_4] = e_4. 
\end{align*}
    \noindent It is not unimodular since $\tr(\ad_{e_1}) = 2$, and it is solvable. Notice that $\eta := \xi^* + e^2 + e^4$ is a contact form on $\g$ with Reeb vector $\xi$, and certainly $\xi = [e_1,e_2] - [\xi,e_1] \in [\g,\g]$. Moreover, $\ad_{\xi}$ is nilpotent since $ \ad_{\xi}(e_1) = e_2 - \xi$ and $\ad_{\xi}(e_j) = 0$ for all $2 \leq j \leq 4$. Notice that $\g$ is just $(\R \ltimes_A \R^2) \oplus \mathrm{aff}(\R)$ with $A = \begin{pmatrix}
        1 & 1 \\
        0 & 1
    \end{pmatrix}$. Here, $\aff(\R)$ is the only 2-dimensional nonabelian Lie algebra, whose Lie brackets are given $[u,v] = v$ with respect to some basis $\{u,v\}$. 
\end{example}
\begin{example} \label{ex: converse of main is false 2}
    Let $\g$  be the Lie algebra with basis $\{\xi, e_1, e_2, e_3, e_4, e_5, e_6\}$ and bracket given~by 
\begin{gather*}
    [\xi,e_4] = -e_1, \quad [\xi,e_5] = e_4, \quad [\xi,e_6] = -e_5,\\
    [e_1,e_5] = e_1, \quad [e_1,e_6] = -2\xi+e_4, \quad [e_2,e_3] = \xi+e_3,\\
    [e_3,e_4] = e_1, \quad [e_3,e_5] = -e_4, \quad [e_3,e_6] = e_5,\\
    [e_4,e_5] = 2\xi+e_4, \quad [e_4,e_6] = -2e_5, \quad [e_5,e_6] = 2 e_6.
\end{gather*}
    \noindent It is not unimodular since $\tr(\ad_{e_2}) = 1$, and it is not solvable. Notice that $\eta = \xi^*$ is a contact form on $\g$ with Reeb vector $\xi$, and certainly $\xi = - \frac{1}{2}([e_1, e_6] - [\xi, e_5]) \in [\g,\g]$. Moreover, $\ad_{\xi}$ is nilpotent since it is strictly upper triangular in the given basis. Notice that $\g$ is just $(\sl(2,\R) \ltimes \R^2) \oplus \mathrm{aff}(\R)$ in the basis $\{X, Y, H, p,q, u,v\}$ given by
\begin{gather*}
    X := \frac{1}{3} e_6, \quad Y := \xi + e_4, \quad H := e_5, \\
    p := \xi - \frac{1}{2} e_4, \quad q := \frac{3}{2} e_1, \quad 
     \quad u := e_2, \quad v := \xi + e_3,
\end{gather*}
    \noindent since  then the nonzero brackets become
\begin{gather*}
    [X,Y] = H, \quad [H,X] = 2X, \quad [H,Y] = -2Y, \quad [u,v] = v, \\
    [X, q] = p, \quad [Y,p] = q, \quad [H, p] = p, \quad [H,q] = -q.
\end{gather*} 
\end{example}   
    \indent The next two examples admit compatible Sasakian structures, as can be seen from \cite[Proposition 4.2]{ACN}. Both are centerless; one is solvable, and the other is not. In fact, the second one is precisely \cite[Examples 4.4 (i)]{ACN} with $k = 1$, and the first one is a modification of \cite[Examples 4.4 (ii)]{ACN} so that $\xi \in [\g,\g]$. 
\begin{example} \label{example: sasakian 1}
    Let $\g$ be the Lie algebra with basis $\{\xi,e_1,e_2,e_3,e_4,e_5,e_6\}$ and bracket given~by
\begin{gather*}
    [e_1,e_2]=e_2+2\xi, \quad [e_3,e_4]=e_4+2\xi, \quad [e_5,e_6]=e_2-e_4+2\xi,\\ 
    [\xi,e_6] = - e_5, \quad [e_2,e_6] = \phantom{+} 2 e_5, \quad [e_4,e_6] = \phantom{+} 2 e_5, \\ 
    [\xi,e_5] = \phantom{+}  e_6,  \quad [e_2,e_5] = -2 e_6, \quad [e_4,e_5] = -2 e_6,\\ 
    [e_1,e_5] = \tfrac{1}{2} e_5, \quad [e_3,e_5] = -\tfrac{1}{2} e_5,\quad [e_1,e_6] = \tfrac{1}{2} e_6, \quad [e_3,e_6] = -\tfrac{1}{2} e_6.
\end{gather*} 
\end{example}
    \noindent It is not unimodular since $\tr(\ad_{e_1}) = 1$, and it is solvable. Following \cite[Proposition 4.2]{ACN}, we can establish that it admits a Sasakian structure for which $\xi$ is the Reeb vector, for which $\xi = \frac{1}{2} ( [e_1,e_2] + [e_3, e_4] - [e_5,e_6] ) \in [\g,\g]$. 
\begin{example} \label{example: sasakian 2}
    Let $\g$ be the Lie algebra with basis $\{\xi, e_1, e_2, e_3, e_4, e_5, e_6\}$ and bracket given~by
\begin{gather*}
    [e_1, e_2] = e_2 + 2 \xi, \quad [e_3, e_4] = e_4 + 2 \xi, \quad [e_5, e_6] = 2 \xi, \\
    [\xi, e_6] = - e_5, \quad [e_2, e_6] = -2 e_5, \quad [e_4, e_6] = -2 e_5, \\
    [\xi, e_5] = \phantom{+} e_6, \quad [e_2, e_5] = -2 e_6, \quad [e_4, e_5] = -2 e_6.
\end{gather*}
    \noindent It is not unimodular since $\tr(\ad_{e_1}) = 1$, and it is not solvable. According to \cite[Examples 4.4 (i)]{ACN}, it admits a Sasakian structure for which $\xi$ is the Reeb vector, which certainly lies in $[\g,\g]$. 
\end{example} 
    \indent Denote by $H_B^*(\g)$ the basic cohomology of $(\g, \eta)$, consisting of forms that are closed (in the usual Chevalley-Eilenberg sense) and $\xi$-basic (meaning both $\iota_{\xi} \alpha = 0$ and $ \mathcal{L}_{\xi} \alpha = 0$) modulo exact $\xi$-basic forms. Notice that $\omega = -d_{\g}\eta$ is $\xi$-basic since $\iota_{\xi}\omega = -\iota_{\xi} d_{\g} \eta = 0$ and $\mathcal{L}_{\xi} \omega = d_{\g} \iota_{\xi} \omega = 0$.     
\begin{lemma} \label{lemma: xi en el conmutador}
    The following conditions are equivalent: 
\begin{enumerate} [\rm (i)]
    \item $\xi \in [\g,\g]$.
    \item Every class in $H^1_{CE}(\g)$ has a $\xi$-basic representative.
    \item The class $[\omega]_B$ of $\omega := - d_{\g} \eta$ in $H^2_B(\g)$ is nonzero.  
\end{enumerate}
\end{lemma}
\begin{proof} \phantom{.} \\
    (i) $\Rightarrow$ (ii): It follows from the identification $H^1_{CE}(\g) = Z^1_{CE}(\g)$, since closed forms vanish on $[\g,\g]$ by definition. \\
    (ii) $\Rightarrow$ (iii): If $[\omega]_B = 0$ then there is some basic $1$-form $\theta$ such that $\omega = d_{\g} \theta$; then, the 1-form $\alpha := \eta + \theta$ is closed and non-exact, and satisfies $\alpha(\xi) = 1$, contradicting (ii). \\
    (iii) $\Rightarrow$ (i): If $\xi \notin [\g,\g]$ then there is a closed $1$-form $\alpha \in \g^*$ such that $\alpha(\xi) = 1$, and thus $\theta := \eta - \alpha$ is $\xi$-basic and $\omega = - d_{\g} \theta$.  
\end{proof}
    \indent Let $\mathrm{pr}_{\h}:\g \to \h$ denote the canonical linear projection associated with the vector space splitting $\g = \R \xi \oplus \h$. For each $x \in \h = \ker \eta$, define $M_x := \mathrm{pr}_{\h} \circ \ad_x \vert_{\h}$. Setting $\nu := \omega^n$, where $\omega := - d_{\g} \eta$ and $\dim \g = 2n+1$, a straightforward computation shows that
\begin{align*}
    \mathcal{L}_{\xi}( \iota_x \nu) = \iota_{[\xi,x]} , \quad d_{\g} (\iota_x \nu) = \mathcal{L}_x \nu = - \tr( M_x ) \nu
\end{align*}
    \noindent for all $x \in \h$. Thus, a form $\iota_x \nu$ is basic if and only if $x \in \ker (\ad_{\xi} \vert_{\h})$, and $d:\Omega_B^{2n-1}(\g) \to \Omega_B^{2n}$ is zero if and only if $\tr(M_x) = 0$ for all $x \in \ker(\ad_{\xi} \vert_{\h})$. Contact Lie algebras $(\g,\eta)$ are said to be \textit{transversely unimodular} when this last property holds; that is, when each $M_x$ is traceless for any $x \in \ker(\ad_{\xi} \vert_{\h})$. Hence, we might say that the transverse Poincaré duality amounts to the equivalence between $H^{2n}_B(\g) \neq 0$ and  transverse unimodularity. 
\begin{corollary}
    If a contact Lie algebra $(\g,\eta)$ is transversely unimodular then $\xi \in [\g,\g]$.
\end{corollary}
\begin{proof}
    While a similar argument as in the proof of Proposition \ref{prop: its converse need not hold} works, it is easier to point out that transverse unimodularity implies that $[\omega^n]_B$ is a nonzero class in $H_B^{2n}(\g)$, which in turn implies $[\omega]_B$ is a nonzero class in $H^2_B(\g)$, from where we get $\xi \in [\g,\g]$ by Lemma \ref{lemma: xi en el conmutador}. 
\end{proof}
    \indent Since $[\xi, x] = 0$ for all $x \in \ker(\ad_{\xi} \vert_{\h})$, the decomposition of $\ad_x$ with respect to the vector space splitting $\R \xi \oplus \h$ is just 
\begin{align*}
    \ad_x
    =
    \begin{pmatrix}
        0 & * \\
        0 & M_x
    \end{pmatrix}, \quad M_x := \mathrm{pr}_{\h} \circ \ad_x \in \End(\h),
\end{align*}    
    \noindent and thus $\tr(\ad_x) = \tr(M_x)$ for all $x \in \ker \eta$. Since transverse unimodularity concerns only those $x \in \h$ lying in $\ker(\ad_{\xi} \vert_{\h})$, we can only conclude that unimodular contact Lie algebras are transversely unimodular, but not the other way around.
    
    \indent The fact that Theorem \ref{thm: main result} gives the failure of transverse unimodularity on the stated hypothesis is of relevance in connection to the two main versions of the Lefschetz conditions proposed for the contact setting: one for the Chevalley-Eilenberg cohomology, related to the one introduced in \cite{Cagliari}, and another for the basic cohomology, as put forward in \cite{KA}. That is to say, since contact Lie algebras where the adjoint action $\ad_{\xi}$ of the Reeb vector is not nilpotent are not transversely unimodular (except for $\sl(2,\R)$ and $\su(2)$), this means that no such Lie algebra can be $0$-Lefschetz in either of the two versions.

\section{A useful decomposition for contact Lie algebras} \label{section: a useful decomposition}

\indent A natural $\ad_{\xi}$-invariant vector space splitting for a given contact Lie algebra $(\g, \eta)$ comes from the Jordan-Chevalley decomposition of $K := \ad_{\xi}$: if $K_s$ and $K_n$ denote the semisimple part and the nilpotent part of $K$ respectively, then
\begin{gather} \label{eq: the useful decomposition}
    \g = \ker K_s \oplus \im K_s
\end{gather}
    \noindent as vector spaces, since $\ker K_s^2 = \ker K_s$ precisely because $K_s$ is semisimple. For~convenience,~set
\begin{align*}
    \h := \ker \eta, \quad \t := \ker K_s, \quad \q := \im K_s, \quad \t_0 := \t \cap \h.
\end{align*}
    \noindent Thus, according to our conventions,  
\begin{align*}
    \g = \t \oplus \q = \R \xi \oplus \t_0 \oplus \q, \quad \h = \t_0 \oplus \q, 
\end{align*}    
    \noindent again as vector spaces. Clearly, $\q \subseteq \h$. As $K_s$ acts as the zero map in the $0$-eigenspace of $K = \ad_{\xi}$, we also get $\xi \in \t$. Since $K_s$ and $K_n$ are commuting polynomials in $K$, the subspaces $\h$, $\t$ and $\q$ are invariant under $K$, $K_s$ and $K_n$. In particular, the restrictions of any of these operators to either $\h$, $\t$, or $\q$ are all well-defined. Clearly, $K \vert_{\h} = K_s \vert_{\h} + K_n \vert_{\h}$ is the Jordan-Chevalley decomposition of $K \vert_{\h}$; the same is true for the restrictions to $\t$ and $\q$. Notice that $K \vert_{\q} = K_s \vert_{\q} \left(\Id_{\q} + (K_s\vert_{\q})^{-1} K_n \vert_{\q} \right)$, so $K \vert_{\q}$ is invertible; in contrast, $K \vert_{\t} = K_n \vert_{\t}$. We shall use all these facts without further mention.
      
    \indent The condition $\t = \R \xi$ holds precisely when $\h = \q$ (that is, $\t \cap \h = 0$), and has the special significance of rendering most of our treatment vacuous. However, up to isomorphism, only two contact Lie algebras fulfill this property.    
\begin{proposition} \label{prop: t = R xi means sl(2;R) or su(2)}
    If $\t = \R \xi$ then $\g$ is isomorphic to either $\su(2)$ or $\sl(2,\R)$.
\end{proposition} 
\begin{proof}
    The statement can be established with a minor variation on the argument used in \cite[Proposition 8.1]{ACN}. We present an alternative, more elementary proof. Notice that the nilradical $\n$ of $\g$ is $K_s$-invariant since it is $K$-invariant. Assume there is some $x \in \n$ whose $\xi$-component is nonzero, and apply the spectral projector onto $\ker K_s = \R \xi$ to show that $\xi \in \n$. As $\n$ is nilpotent, Engel's theorem implies that $\ad_{\xi} \vert_{\n}$ is nilpotent. On the other hand, since $\n = (\n \cap \R\xi) \oplus (\n \cap \h)$ and $K = \ad_{\xi}$ vanishes on $\R\xi$ but is invertible on $\h = \q$, the map $\ad_{\xi} \vert_{\n}$ can be nilpotent only if $\n \cap \h = 0$, meaning $\n = \R\xi$. This situation is impossible, because otherwise
\begin{align*}
    [\xi, \h] = K(\h) = \h \subseteq \n \cap \h = 0,
\end{align*}    
    \noindent contradicting the invertibility of $K\vert_{\h}$. We conclude that every element of $\n$ lies in $\h$. Now, let $x \in \n$. Since $\n$ is an ideal in $\g$ and $\n \subseteq \h$, $[x,y] \in \n \subseteq \h$ for all $y \in \h$, and thus
\begin{align*}
    \iota_x(d_{\g} \eta)(y) = d_{\g} \eta(x,y) = - \eta([x,y]). = 0 \text{ for all $y \in \h$}
\end{align*}    
    \noindent The nondegeneracy of $d_{\g} \eta \vert_{\h}$ implies that $x = 0$, and thus $\n = 0$. Now observe that every nonzero solvable ideal contains a nonzero nilpotent ideal, namely its commutator. Thus, the vanishing of $\n$ implies that the radical of $\g$ is zero. Therefore, $\g$ is semisimple. By \cite[Theorem 5]{BW}, the only semisimple contact Lie
    algebras are $\su(2)$ and $\sl(2,\R)$. 
\end{proof} 
    \indent For the rest of the discussion, we implicitly assume that $\dim \t > 1$. Denote by $\eta \vert_{\t} := \eta_{\t}$ the restriction of $\eta$ to $\t$, and by $\omega_{\t} $ and $\omega_{\q}$ the restrictions of $\omega = - d_{\g} \eta$ to $\t_0$ and $\q$, respectively.
\begin{lemma} \label{lemma: simplecticidad}
    The following identities hold:
\begin{multicols}{3}
\begin{enumerate} [\rm (i)]
    \item $K \vert_{\h}$, $K_s \vert_{\h}\in \sp(\h, \omega)$.
    \item $\ker K_s \vert_{\h} = (\im K_s \vert_{\h} )^{\bot_{\omega}}$.
    \item $\omega(\t_0,\q) = 0$.
    \item $\omega = \omega_{\t} + \omega_{\q}$. 
    \item $\omega_{\t}$ is symplectic on $\t_0$.
    \item $\omega_{\q}$ is symplectic on $\q$. 
\end{enumerate}     
\end{multicols}   
\end{lemma}
\begin{proof} \phantom{.}
\begin{enumerate} [\rm (i)]
    \item For all $u$, $v \in \h$, Jacobi on $(\xi, u,v)$ gives
\begin{align*}
     \omega(K \vert_{\h} u,v) + \omega(u, K \vert_{\h} v) &= - d \eta(K \vert_{\h} u,v) -  d \eta(u, K \vert_{\h}v) = \eta ([[ \xi,u], v] + [u,[\xi,v]] ) \\
    &= \eta ([\xi,[u,v]] ) = d \eta(\xi,[u,v]) = (\iota_{\xi} d \eta)([u,v]) = 0. 
\end{align*} 
    \noindent This shows $K \vert_{\h} \in \sp(\h, \omega)$. The second claim follows from \cite[Theorem 9.20]{FH} applied to the inclusion $\sp(\h,\omega) \hookrightarrow \mathfrak{gl}(\h)$, which is a representation of $\sp(\h,\omega)$. 
    \item For all $u \in \ker K_s \vert_{\h}$ and $v = K_s \vert_{\h} w \in \im K_s \vert_{\h}$, it follows from (i) that
\begin{align*}
    \omega(u,v) = \omega(u, K_s \vert_{\h} w) = - \omega(K_s\vert_{\h} u,w) = - \omega(0,w) = 0. 
\end{align*}    
    \noindent This shows $\ker K_s \vert_{\h} \subseteq (\im K_s \vert_{\h} )^{\bot_{\omega}}$. Equality follows from dimension counting, since
\begin{align*}
    \dim (\im K_s \vert_{\h} )^{\bot_{\omega}} = \dim \h - \dim \im K_s \vert_{\h} = \dim \ker K_s \vert_{\h},
\end{align*} 
    \item It is immediate from (ii), since $\ker K_s \vert_{\h} = \t_0$ and $\im K_s \vert_{\h} = \q$. 
    \item It is immediate from (iii). 
    \item Let $x \in \t_0$ such that $\omega_{\t}(x,\t_0) = 0$. From (iii), we have $\omega(x, \q) = 0$ as well. Thus $\omega(x, \h) = 0$. Since $\omega$ is nondegenerate on $\h$, it follows that $x = 0$. Therefore $\omega_{\t}$ is nondegenerate on $\t_0$.  
    \item It is similar to (v).  \qedhere 
\end{enumerate}    
\end{proof}     
    \indent Since $K$ is a derivation on $\g$, both $K_s$ and $K_n$ are derivations on $\g$: this fact follows from a standard argument using the decomposition of $\g$ into generalized eigenspaces of $K$.  
\begin{lemma} \label{lemma: q is a t-module}
    Both $[\t,\t] \subseteq \t$ and $[\t, \q] \subseteq \q$ hold. 
\end{lemma}
\begin{proof}
    For the first claim, take $x$, $y \in \t$, and note that
\begin{align*}
    K_s([x,y]) = [K_s(x), y] + [x, K_s(y)] = 0.
\end{align*}
    \noindent For the second claim, take $x \in \t$ and $y \in \q$, where $y = K_s(z)$ for some $z \in \g$, and note that 
\begin{align*}
    [x,y] &= [x, K_s(z)] = K_s([x,z]) - [K_s(x),z] = K_s([x,z]) \in \im K_s = \q. \qedhere 
\end{align*} 
\end{proof}
    \indent Compare Lemma \ref{lemma: q is a t-module} with \cite[Proposition 8 (2)]{AFV}. Notice that their proof does not carry over to our setting. 

    \indent Lemma \ref{lemma: simplecticidad} (v) and Lemma \ref{lemma: q is a t-module} imply that $(\t, \eta_{\t})$ is a contact Lie subalgebra of $(\g, \eta)$. Lemma \ref{lemma: simplecticidad} (vi) implies that $\q$ is never a Lie subalgebra of $\g$, for otherwise $[\q,\q] \subseteq \q \subseteq \h = \ker \eta$ and $\eta([\q,\q]) = 0$, meaning that $\omega_{\q} = - (d \eta) \vert_{\q}$ is identically zero. Moreover, using Lemma \ref{lemma: casi todo jacobi} below, it is easy to find conditions ensuring $[\q,\q] \subseteq \t$.

    \indent Consider the skew-symmetric maps
\begin{align*}
    \alpha \colon \alt^2 \t^* \to \t_0, \quad \beta \colon\alt^2 \q^* \to \t_0, \quad \gamma \colon \alt^2 \q^* \to \q,
\end{align*}    
    \noindent defined as the $\t_0$-component of $[\t,\t]$, the $\t_0$-component of $[\q,\q]$, and the $\q$-component of $[\q,\q]$, respectively. A short computation shows that
\begin{gather*}
    [x,y] = \omega_{\t}(x,y) \xi +  \alpha(x,y) \text{ for all $x$, $y \in \t$}, \\
    [u,v] = \omega_{\q}(u,v) \xi + \beta(u,v) + \gamma(u,v) \text{ for all $u$, $v \in \q$}. 
\end{gather*}
    \noindent Indeed, if $\delta_1(\cdot, \cdot) \xi$ and $\delta_2(\cdot,\cdot) \xi$ denote the projection of $[\t,\t]$ and $[\q,\q]$ into $\R\xi$ respectively, then 
\begin{gather*}
    \delta_1(x,y) = \eta([x,y]) = - (d_{\g} \eta)(x,y) = \omega_{\t}(x,y) \text{ for all $x$, $y \in \t$}, \\ 
    \delta_2(u,v) = \eta([u,v]) = - (d_{\g} \eta)(u,v) = \omega_{\q}(u,v) \text{ for all $u$, $v \in \q$}. 
\end{gather*}
    \indent In combination with the fact that $\q = \im K_s$ is $K$-invariant, Lemma \ref{lemma: q is a t-module} allows for the good definition of the following operators:
\begin{align*}
    C := \ad_{\xi} \vert_{\t} \in \End(\t), \quad A := \ad_{\xi} \vert_{\q} \in \End(\q), \quad \rho(x) := \ad_x \vert_{\q} \in \End(\q) \quad (x\in \t_0).
\end{align*} 
    \indent We now derive from the Jacobi identity a number of relations and constraints for all the maps above, only a handful of which are of direct use for establishing Theorem \ref{thm: main result}. For its statement, we employ the usual pullback notation on bilinear forms: for example, if $x$, $y \in \t_0$, then $(C^* \omega_{\t})(x,y)$ means $\omega_{\t}(Cx,y) + \omega_{\t}(x,Cy)$. We also regard $\End(\q)$ as a Lie algebra with the usual commutator: $[M,N] := MN - NM$ for all $M$, $N \in \End(\q)$.    
\begin{lemma} \label{lemma: casi todo jacobi}
    The following identities hold for all $x$, $y$, $z \in \t_0$ and $u$, $v$, $w \in \q$: 
\begin{multicols}{2}
\begin{enumerate}[\rm (i)]
    \item $C^* \omega_{\t} = 0$. That is, $C \in \sp(\t_0,\omega_{\t})$.
    \item $A^* \omega_{\q} = 0$. That is, $A \in \sp(\q,\omega_{\q})$.
    \item $(C^* \alpha)(x,y) = C(\alpha(x,y))$. 
    \item $(A^*\beta)(u,v) = C(\beta(u,v))$.
    \item $(A^* \gamma)(u,v) = A (\gamma(u,v))$.
    \item $[A, \rho(x)] = \rho(Cx)$. 
    \item $[\rho(x), \rho(y)] = \rho( \alpha(x,y)) + \omega_{\t}(x,y) A$.
    \item $(\rho(x)^* \omega_{\q})(u,v) = \omega_{\t}(x, \beta(u,v))$. 
    \item $(\rho(x)^* \beta)(u,v) = \alpha(x, \beta(u,v)) -~\omega_{\q}(u,v) Cx$.
    \item $(\rho(x)^* \gamma)(u,v) = \rho(x)(\gamma(u,v))$.
    \item $\mathfrak{S}_{x,y,z} \; \omega_{\t} (\alpha(x,y),z) = 0$.    
    \item $\mathfrak {S}_{u,v,w} \; \omega_{\q}(\gamma(u,v),w) = 0$. 
    \item $\mathfrak{S}_{u,v,w} \; \beta(\gamma(u,v),w) = 0$. 
    \item $\mathfrak{S}_{x,y,z} \; \alpha(\alpha(x,y),z) = \mathfrak{S}_{x,y,z} \omega_{\t}(x,y) Cz$.  
    \item $\mathfrak{S}_{u,v,w} \gamma(\gamma(u,v),w) ) = \mathfrak{S}_{u,v,w} T(u,v,w)$.
\end{enumerate}
\end{multicols}
    \noindent Here, $\mathfrak{S}$ denotes cyclic sum; also, $T(u,v,w) := \omega_{\q}(u,v)Aw + \rho(\beta(u,v))w$. 
\end{lemma}
\begin{proof} 
    Lemma \ref{lemma: simplecticidad} and Lemma \ref{lemma: q is a t-module} are being implicitly used throughout all the proof.  \\
    \indent For all $x$, $y\in\t_{0}$, Jacobi on $(\xi,x,y)$ gives
\begin{align*}
    C(\alpha(x,y)) &= [\xi, \alpha(x,y)] = [\xi,[x,y]] = [[\xi,x],y] + [x,[\xi,y]] \\
    &= [Cx,y] + [x,Cy] = C^*[x,y] = (C^* \omega_{\t})(x,y) \xi + (C^* \alpha)(x,y).
\end{align*}
    \noindent The $\xi$-component gives (i), and the $\t_0$-component gives (iii). 

    \indent For all $u$, $v \in \q$, Jacobi on $(\xi, u,v)$ gives 
\begin{align*}
   C(\beta(u,v)) + A(\gamma(u,v)) &= [\xi, \beta(u,v) + \gamma(u,v)] = [\xi,[u,v]] = [[\xi,u],v] + [u,[\xi,v]]  \\
   &= [Au,v] + [u,Av] = A^* [u,v] = (A^* \omega_{\q})(u,v) \xi + (A^*\beta)(u,v) + (A^*\gamma)(u,v).
\end{align*}
    \noindent The $\xi$-component gives (ii), the $\t_0$-component gives (iv), and the $\q$-component gives (v). 

    \indent For all $x \in \t_0$ and $u \in \q$, Jacobi on $(\xi,x,u)$ gives 
\begin{align*}
    \rho(Cx)u = [[\xi,x],u] &= [\xi, [x,u]] + [x,[u,\xi]] = A \rho(x) u - \rho(x) Au = [A, \rho(x)] u, 
\end{align*}
    \noindent whence (vi) follows.  

     \indent For all $x$, $y \in \t_0$ and $u \in \q$, Jacobi on $(x,y,u)$ gives 
\begin{align*}
    \rho( \alpha(x,y)) u + \omega_{\t}(x,y) A u &= [\alpha(x,y) + \omega_{\t}(x,y) \xi,u] = [[x,y],u] = [x, [y,u]] + [y, [u,x]] \\
    &= \rho(x)(\rho(y)u) - \rho(y)(\rho(x)u) = [\rho(x), \rho(y)]u, 
\end{align*} 
    \noindent whence (vii) follows. 

    \indent For all $x \in \t_0$ and $u$, $v \in \q$, Jacobi on $(x,u,v)$ gives
\begin{align*}
    [x,\omega_{\q}(u,v)\xi+\beta(u,v)+\gamma(u,v)] &= [x,[u,v]] = [[x,u],v] + [u,[x,v]] = [\rho(x)u,v] + [u,\rho(x)v] \\
    &= (\rho(x)^* \omega_{\q})(u,v)\xi + (\rho(x)^* \beta)(u,v) + (\rho(x)^* \gamma)(u,v).
\end{align*}
    \noindent The LHS is certainly equal to $\omega_{\t}(x,\beta(u,v)) \xi + \alpha(x,\beta(u,v)) - \omega_{\q}(u,v) Cx + \rho(x)\gamma(u,v)$. Thus, the $\xi$-component gives (viii), the $\t_0$-component gives (ix), and the $\q$-component gives (x).

    \indent For all $x$, $y$, $z \in \t_0$, Jacobi on $(x,y,z)$ gives
\begin{align*}
    [x,\omega_{\t}(y,z)\xi+\alpha(y,z)] &= [x,[y,z]] = [[x,y],z] + [y,[x,z]] \\
    &= [\omega_{\t}(x,y)\xi+\alpha(x,y),z] + [y,\omega_{\t}(x,z)\xi+\alpha(x,z)] \\
    &=  (\omega_{\t}(\alpha(x,y),z) + \omega_{\t}(y,\alpha(x,z)) )\xi \\
    & \quad + \alpha(\alpha(x,y),z) + \alpha(y,\alpha(x,z)) + \omega_{\t}(x,y) Cz - \omega_{\t}(x,z) Cy.
\end{align*}
    \noindent The LHS is certainly equal to $\omega_{\t}(x,\alpha(y,z)) \xi + \alpha(x,\alpha(y,z)) - \omega_{\t}(y,z) Cx$. Thus, the $\xi$-component gives (xi), and the $\t_0$-component gives (xiv).  

    \indent For all $u$, $v$, $w \in \q$, Jacobi on $(u,v,w)$ gives
\begin{align*}
    0 &= [u,[v,w]] + [v,[w,u]] + [w,[u,v]] \\
      &= \mathfrak{S}_{u,v,w} ([u,\omega_{\q}(v,w)\xi+\beta(v,w)+\gamma(v,w)] ).
\end{align*}
    \noindent The summand inside $\mathfrak{S}_{u,v,w}$ is certainly equal to
\begin{align*}
    \omega_{\q}(u,\gamma(v,w))\xi + \beta(u,\gamma(v,w)) + \gamma(u,\gamma(v,w))
    - \omega_{\q}(v,w) Au - \rho(\beta(v,w))u.
\end{align*}
    \noindent The $\xi$-component gives (xii), the $\t_0$-component gives  (xiii), and the $\q$-component gives (xv).
\end{proof} 
  \indent Notice that Lemma \ref{lemma: casi todo jacobi} (i) and (ii) follow from $\mathcal{L}_{\xi} \omega = 0$. In fact, 
\begin{gather*}
    0 = (\mathcal{L}_{\xi} \omega)(x,y) = (\ad_{\xi}^* \omega)(x,y) = (C^* \omega_{\t})(x,y) \text{ for all $x$, $y \in \t_0$}, \\
    0 = (\mathcal{L}_{\xi} \omega)(u,v) = (\ad_{\xi}^* \omega)(u,v) = (A^* \omega_{\q})(u,v) \text{ for all $u$, $v \in \q$}.
\end{gather*}
    \noindent Similarly, Lemma \ref{lemma: casi todo jacobi} (i), (ii), (viii), (xi), (xii) all follow from the facts $d_{\g} \omega = 0$ and $\iota_{\xi} \omega = 0$ by evaluating in $(\xi, x,y)$, $(\xi, u,v)$, $(x,u,v)$, $(x,y,z)$, and $(u,v,w)$, respectively. For instance, 
\begin{align*}
    0 &= - (d_{\g} \omega)(x,u,v) = \omega([x,u],v) + \omega([v,x], u) + \omega([u,v], x) \\
    &= \omega_{\q}(\rho(x)u,v) + \omega_{\q}(u, \rho(x)v) - \omega_{\t}(x, \beta(u,v)),
\end{align*}
    \noindent and analogously for the other ones. Interestingly enough, upon evaluating in $(x,y,u)$ we obtain the tautology $0 = 0$. 

    \indent We collect the following immediate consequences of Lemma \ref{lemma: casi todo jacobi} for future reference. 
\begin{corollary} \label{cor: b and b-omega}
    Set $\b := \im \beta$ and $\b^{\omega} := \b^{\bot_{\omega_{\t}}}$. Let $x \in \t_0$. 
\begin{multicols}{2}
\begin{enumerate}[\rm (i)]
    \item Both $\b$ and $\b^{\omega}$ are $C$-stable. 
    \item $[A,\rho(\b^{\omega})]\subseteq \rho(\b^{\omega})$.
    \item $\rho(x) \in \sp(\q,\omega_{\q})$ if and only if $x \in \b^{\omega}$.     
    \item $\tr(\rho(x))=0$ for every $x\in \b^{\omega}$.
\end{enumerate}    
\end{multicols}
\end{corollary}
\begin{proof} \phantom{.}
\begin{enumerate} [\rm (i)]
    \item Since $\im \beta$ is contained in $\im (A^* \beta)$, Lemma \ref{lemma: casi todo jacobi} (iv) shows that $\b$ is $C$-stable. On the other hand, Lemma \ref{lemma: casi todo jacobi} (i) shows that $\b^{\omega}$ is $C$-stable.
    \item It follows from Lemma \ref{lemma: casi todo jacobi}(vi) and (i). 
    \item It is immediate from Lemma \ref{lemma: casi todo jacobi} (viii). 
    \item It follows from (iii), since every symplectic endomorphism is traceless. \qedhere 
\end{enumerate}
\end{proof}
    \indent Lemma \ref{lemma: casi todo jacobi} (ix) can be rephrased as $(\rho(x)^* \beta)(u,v) = [x, \beta(u,v)]$. Were $\t_0$ a Lie algebra, this would imply that $\b$ is an ideal in $\t_0$. This is usually not the case: as Lemma \ref{lemma: casi todo jacobi} (xiv) and (xv) show, in general neither $\alpha$ nor $\gamma$ are Lie brackets on $\t_0$ or $\q$. While we cannot give a clean characterization for when $\gamma$ is a Lie bracket, the situation for $C$ is different.  
\begin{corollary} \label{cor: when is alpha a lie bracket}
    $\alpha$ is a Lie bracket on $\t_0$ if and only if either $\dim \t_0 = 2$ or $C = 0$, the latter of which happens precisely when $\xi$ is central in $\t$. 
\end{corollary}     
\begin{proof}
    According to Lemma \ref{lemma: casi todo jacobi} (iv), $\alpha$ is a Lie bracket on $\t_0$ if and only if the alternating trilinear map defined as $S(x,y,z) := \mathfrak{S}_{x,y,z} \omega_{\t}(x,y) Cz$ vanishes identically. If $\dim \t_0=2$ then this is automatic for dimensional reasons. If $\dim \t_0\geq 4$ then, for a fixed $x\in \t_0$, the dimension of $x^{\bot_{\omega_{\t}}}$ is at least $3$ and so there exist $y$, $z \in \t_0$ such that $\omega_{\t}(x,y) = \omega_{\t}(x,z) = 0$ and $\omega_{\t}(y,z) = 1$, from where it follows that the RHS is just $Cx$. The rest of the claim readily follows. 
\end{proof}
    \indent The situation when $C = 0$, and thus $\alpha$ is a Lie bracket, is discussed in detail in Section \ref{section: DS-contact} below. In there, contact Lie algebras for which $C$ is the zero map are called \textit{DS-contact Lie algebras}. 

    \indent If we are careful to define $\rho(\xi) := A$ then Lemma \ref{lemma: casi todo jacobi} (vii) becomes simply $[\rho(x), \rho(y)] = \rho([x,y])$ for all $x$, $y \in \t$. Together with Lemma \ref{lemma: casi todo jacobi} (vi), this shows that $\r := \rho(\t) \subseteq \End(\q)$ is a Lie subalgebra. The tautological action of $\r$ on $\q$ induces an action on $\alt^2 \q^*$, hence on $\Hom(\alt^2 \q^*, \q)$. By Lemma \ref{lemma: casi todo jacobi} (v) and (x), the map $\gamma \in \Hom(\alt^2 \q^*, \q)$ is $\r$-equivariant. This severely constrains $\gamma$: in fact, for every $D \in \r$, the fact that $D^* \gamma = D\gamma$ gives 
\begin{align*}
    \gamma( E_{\lambda}(D), E_{\mu}(D) ) \subseteq E_{\lambda + \mu}(D) \text{ for all $\lambda \in \spec(D_{\C})$},
\end{align*}
    \noindent where $E_{\lambda}(D)$ denotes the generalized eigenspace of the complexification $D_{\C}$ of $D$ with eigenvalue $\lambda \in \C$. In particular, $\gamma$ is identically zero as soon as there exists $D \in \r$ such that 
\begin{align*}
    \spec(D_{\C}) \cap (\spec(D_{\C}) + \spec(D_{\C})) = \emptyset.
\end{align*}
    \noindent Since Lemma \ref{lemma: casi todo jacobi} (xii), (xiii), and (xv) constrain $\gamma$ even further, one should expect $\gamma$ to vanish in most cases. As the next example shows, this is not always true.      
\begin{example}
    Let $\g$ be the Lie algebra with basis $\{\xi,x,y,e_1,e_2,e_3,e_4\}$ and bracket given by  
\begin{gather*}
    [\xi,e_1] = e_1, \quad [\xi,e_2] = 2e_2, \quad [\xi,e_3] = -e_3, \quad [\xi,e_4] = -2e_4,\\
    [x,e_1] = e_1, \quad [x,e_2] = e_2, 
    [x,y] = \xi+y,\\
    [y,e_1] = -e_1, \quad [y,e_2] = -2e_2, \quad [y,e_3] = e_3, \quad y,e_4] = 2e_4,\\
    [e_1,e_3] = \xi+y, \quad [e_2,e_4] = \xi+y, \quad [e_3,e_2] = e_1.
\end{gather*}
    \noindent Notice $\eta = \xi^*$ is a contact form on $\g$ with Reeb vector $\xi$. It is straightforward to check that, for $\t := \Span\{\xi, x, y\}$ and $\q := \Span \{ e_1, e_2, e_3, e_4\}$, both $\ad_{\xi} \vert_{\t} = 0$ and $A := \ad_{\xi} \vert_{\q} = \diag(1, 2, -1, -2)$. Clearly, $\gamma(e_e, e_2) = e_1$, so $\gamma \neq 0$. 
\end{example}     
    \indent Along the same lines as the last paragraph's remarks, Lemma \ref{lemma: casi todo jacobi} (iii) shows that $\alpha$ is $C$-equivariant, and so it is somewhat restrained. However, we are most interested in the case when $C$ is identically zero, not $\alpha$, as this is the main object of interest in Section \ref{section: DS-contact}. 
    
    \indent Apropos of the previous discussion, the next result can be established by elementary means. 
\begin{corollary} \label{cor: gamma = 0 when dim q = 2}
    If $\dim \q = 2$ then $\gamma = 0$. 
\end{corollary}
\begin{proof}
    If $\dim \q = 2$ then $\dim \alt^2 \q^* = 1$, and therefore $\gamma$ is just a single vector in $\q$. From Lemma \ref{lemma: casi todo jacobi} (ii) and (v), we get $A \gamma = A^* \gamma = \tr(A) \gamma = 0$, and hence $\gamma = 0$ since $A = \ad_{\xi} \vert_{\q}$ is invertible. 
\end{proof}
    \indent Corollary \ref{cor: gamma = 0 when dim q = 2} generalizes the findings in \cite[Section 4]{ACN}, where it is shown by direct computation that $\gamma = 0$ when $\dim \q = 2$ and $(\g,\eta)$ is K-contact. We obtain a ``dual"\! result for $\beta$. 
\begin{corollary} \label{cor: beta neq when dim q geq 4}
    If $\beta = 0$ then $\dim \q = 2$.   
\end{corollary}
\begin{proof}
    Notice that $\k := \R \xi \oplus \q$ is a Lie subalgebra of $\g$ because $\q$ is $\ad_{\xi}$ invariant and $[\q,\q] \subseteq \k$ since $\beta = 0$. Moreover, $\lambda := \eta \vert_{\k}$ is a contact form on $\k$ with Reeb vector $\xi$, since both $\ker \lambda = \q$ and $(d \lambda)\vert_{\q} = \omega_{\q}$ is symplectic. Since $K_s \vert_{\q}$ is invertible by definition of $\q$, it follows that $\ker( K_s \vert_{\k} ) = \R \xi$. Proposition \ref{prop: t = R xi means sl(2;R) or su(2)} gives $\dim \k = 3$, and thus $\dim \q = 2$. 
\end{proof}
    \indent There are examples of contact Lie algebras with $\beta = 0$ and $\dim \q = 2$: $\sl(2,\R)$ and $\su(2)$ fit the bill, but check also Proposition \ref{prop: 5-dimensional DS-contact Lie algebras} below. 

\section{Proof of the main result} \label{section: proof of the main result}

\indent Assume throughout this section that $\q \neq 0$ and $\dim \t \geq 3$ (which, following Proposition \ref{prop: t = R xi means sl(2;R) or su(2)}, merely excludes the Lie algebras $\sl(2,\R)$ and $\su(2)$). Let $\ell \in \g^*$ be the $1$-form on $\g$ defined by
\begin{align} \label{eq: ell form}
    \ell(\xi) = 0,  \quad \ell \vert_{\q} = 0,  \quad \ell(x) := \frac{1}{\dim \q} \tr(A^{-1} \rho(x)) \quad \text{($x \in \t_0$)}, 
\end{align}
    \noindent and let $\sigma \in \alt^2 \t^*$ be the $2$-form on $\t$ defined by 
\begin{align} \label{eq: sigma form}
    \iota_{\xi} \sigma = 0,  \quad \sigma(x,y) := \frac{1}{\dim \q}\tr (A^{-1}[\rho(x),\rho(y)] ) \quad \text{($x,y \in \t_0$)}.
\end{align}
    \noindent Since $[x,y] = \omega_{\t}(x,y)\xi + \alpha(x,y)$ for all $x$, $y \in \t_0$, we have $d_{\t}\ell(x,y) = -\ell([x,y]) = -\ell(\alpha(x,y))$. 
\begin{proposition} \label{prop: propiedades de ell}
    The following identities hold. 
\begin{multicols}{2}
\begin{enumerate} [\rm (i)]
    \item $\ell \circ C = 0$ on $\t_0$.
    \item $d_{\t} \ell = \omega_{\t} - \sigma$ on $\t$.
\end{enumerate}
\end{multicols}
\end{proposition}

\begin{proof} \phantom{.}
\begin{enumerate} [\rm (i)]
    \item Let $x \in \t_0$. Recall from Lemma \ref{lemma: casi todo jacobi} (vi) that $[A,\rho(x)] = \rho(Cx)$, from where it follows that 
\begin{align*}
    (\dim \q)\,\ell(Cx)
    = \tr (A^{-1}\rho(Cx) )
    = \tr (A^{-1}[A,\rho(x)] )
    = \tr (\rho(x)-A^{-1}\rho(x)A )=0.
\end{align*}
    \item Let $x$, $y \in \t_0$. Recall from Lemma \ref{lemma: casi todo jacobi} (vii) that $[\rho(x), \rho(y)] = \rho(\alpha(x,y)) + \omega_{\t}(x,y) A$. Multiplying this identity by $A^{-1}$ and taking trace, we obtain  
 \begin{align*}
     \tr(A^{-1} [\rho(x), \rho(y)]) = (\dim \q) \omega_{\t}(x,y) + \tr(A^{-1} \rho(\alpha(x,y))),
 \end{align*}
    \noindent hence the identity holds on $\alt^2 \t_0^*$. It also holds in $\R \xi \wedge \t$, because
\begin{align*}
    (d_{\t} \ell)(\xi,x) &= -\ell([\xi,x]) = -\ell(Cx)=0 \quad \text{and} \quad \omega_{\t}(\xi,x) = \sigma(\xi,x)=0. \qedhere 
\end{align*}   
\end{enumerate}
\end{proof} 
    \indent From Lemma \ref{lemma: simplecticidad} (v) we know that $\omega_{\t}$ is nondegenerate on $\t_0$. Thus, there exists a unique $e \in \t_0$ satisfying $\ell = - \iota_e \omega$ on $\t_0$. Since $[\t,\t] \subseteq \t$ and $[\t,\q] \subseteq \q$, as guaranteed by Lemma \ref{lemma: q is a t-module}, the matrix $\ad_e$ with respect to the vector space splitting $\g = \t \oplus \q$ is 
\begin{align*}
    \ad_e = 
    \begin{pmatrix}
        \ad_e \vert_{\t} & 0 \\
        0 & \ad_e \vert_{\q}
    \end{pmatrix}.
\end{align*}
    \noindent Since $\t$ is a subalgebra of $\g$, $\ad_e \vert_{\t}$ represents both the restriction of $\ad_e$ to $\t$ and the adjoint operator of $e$ inside $\t$. In particular,
\begin{align} \label{eq: la estrategia}
    \tr(\ad_e) = \tr(\ad_e \vert_{\t}) + \tr(\ad_e \vert_{\q}). 
\end{align}
    \noindent Our goal is to show that the first term is positive, and that the second is nonnegative. 
\begin{lemma} \label{lemma: se usa luego}
    If $z \in \ker C$ then $\iota_z \sigma = 0$. In particular, both $[\xi, e] = 0$ and $\iota_e \sigma = 0$.  
\end{lemma}
\begin{proof}
    According to Lemma \ref{lemma: casi todo jacobi} (vi), if $z \in \ker C$ then $[A, \rho(z)] = 0$, and so
\begin{align*}
    A^{-1}[\rho(z),\rho(x)] = [A^{-1}\rho(z),\rho(x)] \text{ for all $x \in \t_0$}. 
\end{align*}
    \noindent Since the trace of a commutator vanishes, we obtain 
\begin{align*}
    (\iota_z \sigma)(x) = \sigma(z,x) &= \frac{1}{\dim\q} \tr([A^{-1}\rho(z),\rho(x)]) = 0 \text{ for all $x \in \t_0$}. 
\end{align*}
    \noindent For the second part, notice that Proposition \ref{prop: propiedades de ell} (i) implies 
\begin{align*}
    0 = \ell(Cx) = - \omega(e,Cx) = \omega(Ce,x) \text{ for all $x \in \t_0$}, 
\end{align*}  
    \noindent thus the nondegeneracy of $\omega$ forces $0 = Ce = [e, \xi]$, hence $\iota_e \sigma = 0$ from the first part.   
\end{proof} 
\begin{proposition} \label{prop: el primer término es positivo}
    $\tr(\ad_e \vert_{\t}) = \frac{\dim \t_0}{2} > 0$. In particular, $\t$ is not unimodular.  
\end{proposition} 
\begin{proof}
     The first part of Lemma \ref{lemma: se usa luego} implies that 
\begin{align*}
    \ad_e \vert_{\t}
    =
    \begin{pmatrix}
        0 & * \\
        0 & M
    \end{pmatrix}, \quad M := \mathrm{pr}_{\t_0}\circ \ad_e \vert_{\t_0} \in \End(\t_0),
\end{align*}    
    \noindent where $\mathrm{pr}_{\t_0}:\t \to \t_0$ denotes the canonical linear projection with respect to the vector space splitting $\t = \R \xi \oplus \t_0$. In particular, $\tr(\ad_e\vert_{\t}) = \tr(M)$. Since $\omega_{\t}$ is nondegenerate on $\t_0$, there exists a unique endomorphism $N \in \End(\t_0)$ satisfying
\begin{align*}
    \sigma(x,y) = \omega_{\t}(Nx,y) \text{ for all $x$, $y \in \t_0$}.
\end{align*}
    \noindent Notice that $N \in \sp(\t_0,\omega_{\t})$ since $\sigma$ is skew-symmetric. By Proposition \ref{prop: propiedades de ell} (ii),  we get 
\begin{align*}
    \mathcal{L}_e \omega_{\t} = d_{\t}(\iota_e\omega_{\t})= -d_{\t}\ell = -\omega_{\t} + \sigma,
\end{align*}
    \noindent which is to say
\begin{align*}
    -\omega_{\t}(Mx,y)-\omega_{\t}(x,My) = -\omega_{\t}(x,y)+\omega_{\t}(Nx,y) \text{ for all $x$, $y \in \t_0$}.
\end{align*}
    \noindent Equivalently, $M - \frac{1}{2}(\Id - N)\in \sp(\t_0,\omega_{\t})$. Therefore, both $N$ and $M - \frac{1}{2} (\Id - N)$ are symplectic, and therefore traceless. Hence, 
\begin{align*}
    \tr(\ad_e \vert_{\t}) &= \tr(M) = \frac{1}{2} \tr(\Id-N) = \frac{\dim \t_0}{2}. \qedhere 
\end{align*} 
\end{proof} 
    \indent Let $\g_{\C}$ be the complexification of $\g$. Decompose $\q_{\C}$ into generalized eigenspaces $E_{\lambda}$ of $A_{\C}$ as
\begin{align*}
    \q_{\C} = \bigoplus_{\lambda \in \spec(A_{\C})} E_{\lambda}. 
\end{align*} 
    \indent Since $A$ is invertible, $0 \notin \spec(A_{\C})$. Set $\Omega_{\q} := - \ell \circ \beta$, viewed on $\g_{\C}$. We omit the symbol $\C$ to avoid overloading the notation. 
\begin{lemma} \label{lemma: generalized eigenspaces}
    Let $\lambda$, $\mu \in \C^{\times}$ such that $\lambda + \mu \neq 0$. 
\begin{multicols}{2}
\begin{enumerate} [\rm (i)]
    \item $\omega_{\q}(E_{\lambda}, E_{\mu}) = 0$. 
    \item $\Omega_{\q}(E_{\lambda}, E_{\mu}) = 0$.
\end{enumerate}    
\end{multicols}
\end{lemma}
\begin{proof} 
    Let $A = A_s + A_n$ be the Jordan-Chevalley decomposition of $A$. Recall $A \in \mathfrak{sp}(\q, \omega_{\q})$ from Lemma \ref{lemma: simplecticidad} (i). Notice that the inclusion $\sp(\q, \omega_{\q}) \hookrightarrow \End(\q)$ and the map 
\begin{align*}
    \pi: \sp(\q,\omega_\q) \to \End(\Hom(\alt^2\q,\t_0)), \quad (\pi(D)\theta)(x,y) := \theta(Dx,y) + \theta(x,Dy)
\end{align*}
    \noindent define Lie algebra representations of $\sp(\q,\omega_{\q})$. Combining Lemma \ref{lemma: casi todo jacobi} (ii) and Proposition \ref{prop: propiedades de ell} (i), we get from the definitions of $\Omega$ and $\pi$ that $\pi(A)\Omega = 0$. Applying \cite[Theorem 9.20]{FH} to each representation, we obtain that $A_s \in \sp(\q, \omega_{\q})$ and that $\pi(A_s) \Omega = 0$. 
\begin{enumerate} [\rm (i)]
    \item Since $A_s \vert_{E_{\lambda}} = \lambda \Id$, for $x \in E_{\lambda}$ and $y \in E_{\mu}$ we get
\begin{align*}
    \omega_{\q}(x,y) &= \frac{1}{\lambda} \omega_{\q}(A_s x, y) = - \frac{1}{\lambda} \omega_{\q}(x, A_s y) = - \frac{\mu}{\lambda} \omega_{\q}(x,y).  
\end{align*}
    \item For $x \in E_{\lambda}$ and $y \in E_{\mu}$, we get 
\begin{align*}
    0 &= (\pi(A_s) \Omega_{\q})(x,y) = \Omega_{\q}(A_s x,y) + \Omega_{\q}(x,A_s y) = (\lambda +\mu) \Omega_{\q}(x,y). \qedhere 
\end{align*} 
\end{enumerate}    
\end{proof} 
    \indent Let $\spec(A_{\C})_+$ be the subset in $\spec(A_{\C})$ of either positive real part or zero real part and positive imaginary part. For the next part, we work with a single $\lambda \in \spec(A_{\C})_+$ at a time. Since $\omega_{\q}$ is nondegenerate on $\q$, it induces a nondegenerate pairing $\omega_{\q}:E_{\lambda} \times E_{-\lambda} \to \C$ for each $\lambda \in \spec(A_{\C})$.
\begin{remark} \label{obs: dim ker (alt^2 A) leq m^2}
    The nondegeneracy of $\omega_{\q}$ on each $E_{\lambda} \times E_{-\lambda}$ implies that $\dim E_{\lambda} = \dim E_{-\lambda}$ for all $\lambda \in \spec(A_{\C})$. Let's call $r_{\lambda}$ that common dimension. Denote by $\alt^2 A :\alt^2 \q^* \to \alt^2 \q^*$ the map given by $(\alt^2 A)(u,v) := Au \wedge v + u \wedge Av$. The semisimple part of the complexification of $\alt^2 A$ acts on $E_{\lambda} \wedge E_{\mu}$ by $\lambda + \mu$, and so $\ker \alt^2 A$ lies on $\bigoplus_{\lambda \in \spec(A_{\C})_+} E_{\lambda} \wedge E_{-\lambda}$. In particular, 
\begin{align*}
    \dim \ker (\alt^2 A) \leq \sum_{\lambda \in \spec(A_{\C})} r_{\lambda}^2 \leq \left( \frac{\dim \q}{2} \right)^2. 
\end{align*}
\end{remark}
    \indent The nondegeneracy of $\omega_{\q}$ on $E_{\lambda} \times E_{-\lambda}$ implies that there exists a unique $T_{\lambda} \in E_{\lambda}$ such that
\begin{align} \label{eq: ecuacion maestra}
    \omega(e, \beta(u,v)) = \Omega_{\q}(u,v) = \omega_{\q}(T_{\lambda} u,v) \text{ for all $u \in E_{\lambda}$ and $v \in E_{-\lambda}$}.
\end{align}
    \noindent Combining Lemma \ref{lemma: casi todo jacobi} (vi) and Lemma \ref{lemma: se usa luego}, we see that $A$ and $\rho(e) = \ad_e \vert_{\q}$ commute; in particular, $\ad_e$ preserves the generalized eigenspaces of $A$. Hence, the operators 
\begin{align*}
    P_{\lambda} := \ad_e \vert_{E_{\lambda}} \in \End(E_{\lambda}), \quad Q_{\lambda} := \ad_e \vert_{E_{-\lambda}} \in \End(E_{-\lambda})
\end{align*}
    \noindent are well-defined. Denote by $Q_{\lambda}^{\dagger} \in \End(E_{-\lambda})$ the $\omega_{\q}$-adjoint of $Q_{\lambda}$.
\begin{lemma} \label{lemma: formulas for T, P, and Q}
    \phantom{.}
\begin{multicols}{2}
\begin{enumerate} [\rm (i)]
    \item $T_{\lambda} = P_{\lambda} + Q_{\lambda}^{\dagger}$.
    \item $T_{\lambda} = T_{\lambda} P_{\lambda} + Q_{\lambda}^{\dagger} T_{\lambda}$. 
\end{enumerate}
\end{multicols}
\end{lemma}
\begin{proof} \phantom{.}
\begin{enumerate} [\rm (i)]
    \item Take $u \in E_{\lambda}$ and $v \in E_{-\lambda}$. Recall that $d_{\g}\omega = 0$ and $\iota_{\xi} \omega = 0$, that $\omega(\gamma(u,v),e) = 0$ due to Lemma \ref{lemma: simplecticidad} (iii), and that $\omega_{\q}(T_{\lambda} u,v) =  \omega(e, \beta(u,v))$ as observed in equation \eqref{eq: ecuacion maestra}. Then, 
\begin{align*}
    0 &= - (d_{\g}\omega)(e,u,v) = \omega([e,u],v) + \omega(u,[e,v]) + \omega([u,v],e) \\
    &= \omega(P_{\lambda} u ,v) + \omega(u, Q_{\lambda} v) + \omega( \omega_{\q}(u,v) \xi + \beta(u,v) + \gamma(u,v), e) \\
    &= \omega(P_{\lambda} u ,v) + \omega(Q_{\lambda}^{\dagger} u, v) - \omega(e, \beta(u,v)) = \omega_{\q}(P_{\lambda} u ,v) + \omega_{\q}(Q_{\lambda}^{\dagger} u, v) - \omega_{\q}(T_{\lambda}u,v).
\end{align*}
    \indent The claim follows from the nondegeneracy of $\omega_{\q}$ on $E_{\lambda} \times E_{-\lambda}$. 
    \item Take $u \in E_{\lambda}$ and $v \in E_{-\lambda}$. Combining  Lemma \ref{lemma: casi todo jacobi} (ix) and Lemma \ref{lemma: se usa luego}, we obtain
\begin{align*}
    \beta(P_{\lambda}u,v) + \beta(u,Q_{\lambda} v) = \alpha(e, \beta(u,v)).
\end{align*}
    \noindent Now apply $- \ell$ to both sides. The resulting expression can be simplified with the definition of $\Omega_{\q}$, equation \eqref{eq: ecuacion maestra}, Proposition \ref{prop: propiedades de ell} (i), and Lemma \ref{lemma: se usa luego} again: 
\begin{align*}
    \omega_{\q}(T_{\lambda} P_{\lambda} u,v) + \omega_{\q}(Q_{\lambda}^{\dagger} T_{\lambda}u,v) &= \Omega_{\q}(P_{\lambda} u, v) + \Omega_{\q}(u, Q_{\lambda}v) = - \ell( [e, \beta(u,v)]_{\t_0}) =(d_{\t} \ell)(e, \beta(u,v)) \\
    &= \omega_{\t}(e, \beta(u,v)) - \sigma(e, \beta(u,v)) = \Omega_{\q}(u,v) = \omega_{\q}(T_{\lambda} u,v). 
\end{align*}
    \noindent The claim follows from the nondegeneracy of $\omega_{\q}$ on $E_{\lambda} \times E_{-\lambda}$. \qedhere 
\end{enumerate}
\end{proof}
    \indent Lemma \ref{lemma: formulas for T, P, and Q} gives valuable information about the spectrum of $T_{\lambda}$. 
\begin{proposition} \label{prop: espectro de T}
    Both $\spec(T_{\lambda}) \subseteq \{0,1\}$ and $\tr(T_{\lambda}) \geq 0$. 
\end{proposition}
\begin{proof}
    Substitute the formula in Lemma \ref{lemma: formulas for T, P, and Q} (i) into the formula in Lemma \ref{lemma: formulas for T, P, and Q} (ii) to obtain
\begin{align*}
    T_{\lambda} = T_{\lambda} P_{\lambda}+(T_{\lambda} - P_{\lambda}) T_{\lambda} = T_{\lambda}^2 + [T_{\lambda},P_{\lambda}].
\end{align*}
    \noindent Let $\mu \in \spec(T_{\lambda})$ be an eigenvalue of $T_{\lambda}$ with generalized eigenspace $W_{\mu}$ and spectral projector $\Pi_{\mu}$. Define $T_{\lambda, \mu} := T_{\lambda} \vert_{W_{\mu}}$ and $P_{\lambda, \mu} := \Pi_{\mu} P_{\lambda} \vert_{W_{\mu}}$. Since $\Pi_{\mu}$ and $T_{\lambda}$ commute, we get $T_{\lambda, \mu} - T_{\lambda, \mu}^2 = [T_{\lambda, \mu},P_{\lambda, \mu}]$, hence $\tr(T_{\lambda, \mu}) = \tr(T_{\lambda, \mu}^2)$.  Since $T_{\lambda, \mu} = \mu \Id + N$ with $N$ nilpotent, clearly
\begin{align*}
    (\dim W_{\mu}) \mu = \tr(T_{\lambda, \mu}) = \tr(T_{\lambda, \mu}^2) = (\dim W_{\mu}) \mu^2,
\end{align*}
    \noindent and so $\mu = \mu^2$, which  implies $\mu \in \{0,1\}$. 
\end{proof}
    \indent Just as $\q_{\C}$ is decomposed as a direct sum of pairs of generalized eigenspaces $E_{\lambda} \oplus E_{-\lambda}$, where $\lambda \in \spec(A_{\C})_+$, the operator $\rho(e) = \ad_e \vert_{\q}$ can be decomposed into a direct sum of blocks of the form $P_{\lambda} \oplus Q_{\lambda}$, again with $\lambda \in \spec(A_{\C})_+$. Recall that the trace of an operator is preserved under complexification, as well as $\tr(Q_{\lambda}) = \tr(Q_{\lambda}^{\dagger})$.  
\begin{proposition} \label{prop: el segundo término es nonegativo}
    $\tr(\ad_e \vert_{\q}) \geq 0$. 
\end{proposition}
\begin{proof}
    Following Proposition \ref{prop: espectro de T}, 
\begin{align*}
    \tr(\ad_e \vert_{\q}) &= \sum_{\spec(A_{\C})_+} \tr(P_{\lambda}) + \tr(Q_{\lambda}) = \sum_{\spec(A_{\C})_+} \tr(P_{\lambda}) + \tr(Q_{\lambda}^{\dagger}) = \sum_{\spec(A_{\C})_+} \tr(T_{\lambda}) \geq 0. \qedhere 
\end{align*}
\end{proof} 
    \indent Our main result follows by tying everything up.  
\begin{theorem} \label{thm: resultado principal}
    If $\t_0 \neq 0$ and $\q \neq 0$ then $\g$ is not transversely unimodular. 
\end{theorem}
\begin{proof}
    Regarding equation \eqref{eq: la estrategia}, the first term is positive by Proposition \ref{prop: el primer término es positivo} and the second term is nonnegative by Proposition \ref{prop: el segundo término es nonegativo}. Thus, $\tr(\ad_e) > 0$. Since Lemma \ref{lemma: se usa luego} gives $[\xi, e] = 0$, we get $e \in \ker(\ad_{\xi} \vert_{\h})$ and thus
\begin{align*}
    \ad_e = 
    \begin{pmatrix}
            0 & *\\
            0 & M_e
        \end{pmatrix}, \quad M_e := \mathfrak{pr}_{\h} \circ \ad_e \vert_{\h} \in \End(\h)
\end{align*}
    \noindent with respect to the vector space splitting $\g = \R \xi \oplus \h$. Therefore, $\tr(M_e) = \tr(\ad_e) > 0$. 
\end{proof}

    \indent By Proposition \ref{prop: t = R xi means sl(2;R) or su(2)}, the condition $\t_0 \neq 0$ only excludes the contact Lie algebras $\sl(2,\R)$ and $\su(2)$. On the other hand, the condition $\im (K_s) = \q \neq 0$ means that $K = \ad_{\xi}$ has nonzero semisimple part. Thus, Theorem \ref{thm: resultado principal} is equivalent to the fact that transversely unimodular contact Lie algebras not isomorphic to $\sl(2,\R)$ or $\su(2)$ must have nilpotent $\ad_{\xi}$. From both assertions we get the claim in Theorem \ref{thm: main result}. Conversely, it is not difficult to find unimodular contact Lie algebras with nilpotent $\ad_{\xi}$, as the next example shows. 
\begin{example} \label{ex: Diatta-Foreman}
    Let $\g$ be the Lie algebra with basis $\{e_1, e_2, e_3, e_4, e_5\}$ and bracket given by
\begin{gather*}
    [e_1, e_4] = e_1, \quad [e_3, e_4] = - e_3, \quad [e_2, e_5] = e_2, \quad [e_3, e_5] = - e_3.
\end{gather*}
    \noindent Here, $\eta := e^1 + e^2 + e^3$ is the contact form and $\xi := \frac{1}{3} (e_1 + e_2 + e_3)$. A direct computation shows $\ad_{\xi}^2 = 0$. As pointed out in \cite[Proposition 4.4]{DF}, it is solvable and the associated simply connected Lie group has lattices. 
\end{example}   
    \indent  As already discussed in Section \ref{section: cohomology remarks}, Theorem \ref{thm: resultado principal} entails that any contact Lie algebra with both $\t_0 \neq 0$ and $\q \neq 0$ fails to be $0$-Lefschetz in both the sense of \cite{Cagliari} and the sense of \cite{KA}. 

\section{DS-contact Lie algebras} \label{section: DS-contact}

    \indent A \textit{DS-contact} Lie algebra is a contact Lie algebra $(\g,\eta)$ for which the following vector-space decomposition holds:  
\begin{align} \label{eq: DS-contact equation again}
    \g = \ker \ad_{\xi} \oplus \im \ad_{\xi}. 
\end{align}
    \noindent The ``DS"\! bit stands for ``direct sum"\!. The validity of such a decomposition is just a constraint on the Jordan form of $\ad_{\xi}$. There are multiple ways to express the~DS-condition,~namely: 
\begin{multicols}{2}
\begin{itemize}
    \item $ \ker \ad_{\xi} \cap \im \ad_{\xi} = \{0\}$.
    \item $\ker \ad_{\xi}^2 = \ker \ad_{\xi}$.
    \item $K_n \vert_{\t} = 0$.
    \item $C = 0$. That is, $\xi$ is central in $\t_0$. 
    \item The characteristic polynomial of $\ad_{\xi}$ is divisible by $x$ but not by $x^2$.
\end{itemize}    
\end{multicols} 
    \indent As observed in the first lines in Section \ref{section: a useful decomposition}, both summands in the vector space decomposition $\g = \t \oplus \q$ in equation \eqref{eq: the useful decomposition} are $K$-invariant, the equality $K \vert_{\t} = K_n \vert_{\t}$ holds, and $K \vert_{\q}$ is invertible. In particular, $\im K = \im K_s \oplus \im K_n \vert_{\t}$ and $\q = \im \ad_{\xi}$ if and only if $K_n \vert_{\t} = 0$. Thus, the decomposition in equation \eqref{eq: DS-contact equation again} is just
\begin{align*}
    \g = \ker \ad_{\xi} \oplus \im \ad_{\xi} = \t \oplus \q,
\end{align*}
    \noindent without any clash of notation whatsoever. 
    
    \indent Notice that K-contact and Sasakian Lie algebras are DS-contact, as noted in the discussion leading up to equation \eqref{eq: DS-contact equation}. Certainly, there are strict DS-contact Lie algebras; that is, ones not admitting compatible K-contact metrics: as seen from Proposition \ref{prop: 5-dimensional DS-contact Lie algebras} below, the Lie algebra labeled $\g_{1,-}$ is strict DS-contact. We emphasize once more that, in contrast with the K-contact case, the decomposition in equation \eqref{eq: DS-contact equation again} is not assumed to be neither $g$-orthogonal nor $\Phi$-invariant for any compatible metric $g$ or $(1,1)$ tensor $\Phi$, although it still is $\ad_{\xi}$-invariant. 

    \indent It is apparent from the definition that, besides $\sl(2,\R)$ and $\su(2)$, the only DS-contact Lie algebras having nilpotent $\ad_{\xi}$ are those with $\ad_{\xi} = 0$, thus having a nontrivial center. In combination with Theorem \ref{thm: main result}, it follows that centerless unimodular DS-contact Lie algebras are either $\sl(2,\R)$ or $\su(2)$. Since K-contact Lie algebras are DS-contact, we obtain a generalization of Corollary \ref{cor: main corollary}.  

\subsection{Elementary properties of DS-contact Lie algebras} \label{section: elementary properties}

    \indent The main motivation behind the introduction of DS-contact Lie algebras is the observation that a number of known properties of Sasakian Lie algebras are consequences of equation \eqref{eq: DS-contact equation again} alone, not the compatible metric structure nor its good behavior. We see this already with the discussion in the previous paragraph, itself a partial generalization of \cite[Theorem 2.1]{AHK}, and with Proposition \ref{prop: main proposition}, which generalizes \cite[Proposition 8.1]{ACN}. Something similar happens with the proof of the next result. 
\begin{lemma}
    Both $[\t,\t] \subseteq \t$ and $[\t,\q] \subseteq \q$ hold. 
\end{lemma} 
\begin{proof}
    While both claims follow from Lemma \ref{lemma: q is a t-module}, we point out that now the proof of $[\t,\q] \subseteq \q$ in \cite[Proposition 8 (2)]{AFV} also holds in this setting. We reproduce it here: Take $x \in \t$ and $y \in \q$, where $y = T z$ for some $z \in \h$. The Jacobi identity on the triple $(x,\xi,z)$ and the fact that $[x,\xi] = 0$ ensure that
\begin{align*}
    [x,y]_{\g} &= [x, [\xi, z]_{\g}]_{\g} = [[x,\xi]_{\g}, z]_{\g} + [\xi, [x,z]_{\g}]_{\g} = [\xi, [x,z]_{\g}]_{\g} = T([x,z]_{\g}) \in \q. \qedhere 
\end{align*}
\end{proof}
    \indent Just as in Lemma \ref{lemma: simplecticidad}, $\t_0$ and $\q$ are $\omega$-orthogonal, from where it follows that the restrictions $\omega_{\t}$ and $\omega_{\q}$ are symplectic on $\t_0$ and $\q$. In particular, $(\t, \eta _{\t})$ is a contact Lie subalgebra of $(\g, \eta)$; moreover, since $\xi \in \t$, it also has nontrivial center. According to Proposition \ref{prop: contactization}, $(\t, \eta _{\t})$ is the contactization of $(\t_0, \omega_{\t})$, and thus $\alpha$ is a Lie bracket on $(\t_0, \omega_{\t})$. Notice that if further $(\g, \eta)$ is K-contact or Sasakian, so is $(\t, \eta _{\t})$. This is observed in \cite[Proposition 8]{AFV}. 

     \indent The following properties all follow from Lemma \ref{lemma: casi todo jacobi} together with the fact that $C = \ad_{\xi} \vert_{\t}$ is the zero map. For the other identities in the same result we obtain no obvious simplifications in the DS-contact setting. 
\begin{lemma} \label{lemma: casi todo jacobi DS}
    The following identities hold in DS-contact Lie algebras, where $x$, $y$, $z \in \t_0$ and $u$, $v \in \q$: 
\begin{multicols}{2}
\begin{enumerate}[\rm (i)]
    \item $(A^*\beta)(u,v) = 0$. 
    \item $[A, \rho(x)] = 0$.   
    \item $(\rho(x)^* \beta)(u,v) = \alpha(x, \beta(u,v))$.
    \item $\alpha(x, \alpha(y,z)) = 0$. 
\end{enumerate}
\end{multicols} 
\end{lemma} 
    \indent Notice that Lemma \ref{lemma: casi todo jacobi DS} (iv) is just the statement that $\alpha$ is a Lie bracket on $\t_0$, which was observed a couple of lines above. While the statement in Lemma \ref{lemma: casi todo jacobi} (xi) remains unchanged in the DS-setting, its meaning does not: it says that $\omega_{\t}$ is closed with respect to $\alpha$, and thus a symplectic form. Theorem \ref{thm: t_0 is frobenius} below ensures that it is exact.    

    \indent Recall from Corollary \ref{cor: b and b-omega} that $\b := \im \beta$ and $\b^{\omega} := \b^{\bot_{\omega_{\t}}}$. We know from Corollary \ref{cor: beta neq when dim q geq 4} that $\b \neq 0$ when $\dim \q \geq 4$. The spaces $\b$ and $\b^{\omega}$ possess a number of extra properties in the DS-contact setting. 
\begin{corollary} \label{cor: some obstructions}
    \phantom{.}
\begin{enumerate} [\rm (i)]
    \item $\b$ is an ideal in $(\t_0, \alpha)$ of dimension $\dim \b \leq \left( \frac{\dim \q}{2} \right)^2$.
    \item $\b^{\omega}$ is a Lie subalgebra of $(\t_0, \alpha)$ of codimension $\mathrm{codim} \, \b^{\omega} \leq \left( \frac{\dim \q}{2} \right)^2$. 
    \item $\rho(\b^{\omega})$ lies in the centralizer of $A$ in $\sp(\q, \omega_{\q})$. 
    \item $\k := \b \cap \b^{\omega}$ is an ideal of $\b^{\omega}$ and $\omega_{\t}$ descends to a symplectic form on $\b^{\omega}/\k$.
    \item The $1$-form $\mu \in \t_0^*$ given by $\mu(x) := \tr(\rho(x))$ is closed and $\b^{\omega} \subseteq \ker \mu$. 
\end{enumerate}
\end{corollary}
\begin{proof} \phantom{.}
\begin{enumerate} [\rm (i)]
    \item Lemma \ref{lemma: casi todo jacobi DS} (iv) ensures $(\t_0, \alpha)$ is a Lie algebra, and Lemma \ref{lemma: casi todo jacobi DS} (iii) ensures $\b$ is an ideal in it. Lemma \ref{lemma: casi todo jacobi DS} (i) implies that $\beta$ vanishes on the image of the induced map $\alt^2 A$; from here, the rank-nullity theorem and Remark \ref{obs: dim ker (alt^2 A) leq m^2} imply
\begin{align*}
    \dim \b = \dim \im \beta \leq \dim \ker (\alt^2 A) \leq \left( \frac{\dim \q}{2} \right)^2. 
\end{align*} 
    \item Let $x$, $y \in \b^{\omega}$ and $z \in \b$. According to (i), both $[y,z] \in \b$ and $[z,x] \in \b$. Since $\omega_{\t}$ is closed, 
\begin{gather*}
    0 = - (d_{\t_0} \omega_{\t}) (x,y,z) = \omega_{\t}(\alpha(x,y),z) + \omega_{\t}(\alpha(y,z),x) + \omega_{\t}(\alpha(z,x),y) = \omega_{\t}(\alpha(x,y),z)
\end{gather*}
    \noindent implying $\alpha(x,y) \in \b^{\omega}$. Thus $\b^{\omega}$ is a Lie subalgebra of $\t_0$. The bound for $\mathrm{codim} \, \b^{\omega}$ follows from the nondegeneracy of $\omega_{\t}$ and (i).
    \item The first assertion follows from Lemma \ref{lemma: casi todo jacobi DS} (ii), while the second assertion arises from here and Corollary \ref{cor: b and b-omega} (iii). 
    \item $\k$ is an ideal in $\b$ because $\b$ is an ideal in $\t_0$ and $\b^{\omega}$. On the other hand, $\omega_{\t}$ descends to a nondegenerate closed $2$-form on $\b^{\omega}/\k$ because $\ker (\omega_{\t} \vert_{\b^{\omega}} ) = \b^{\omega} \cap \b = \k$. 
    \item Lemma \ref{lemma: casi todo jacobi} (ii) ensures that $A \in \sp(\q, \omega_{\q})$, and in particular $\tr(A) = 0$, and Lemma \ref{lemma: casi todo jacobi} (vii) ensures that $[\rho(x),\rho(y)] = \rho(\alpha(x,y))+ \omega_{\t}(x,y) A$ for all $x$, $y \in \t_0$. Hence,  
\begin{align*}
    (d_{\t_0} \mu)(x,y) = - \mu(\alpha(x,y)) = \tr( \rho(\alpha(x,y) ) = \tr( [\rho(x), \rho(y)]) - \omega_{\t}(x,y) \tr(A) = 0,
\end{align*}
    \noindent for all $x$, $y \in \t_0$. The inclusion $\b^{\omega} \subseteq \ker \mu$ follows from Corollary \ref{cor: b and b-omega} (iv). \qedhere 
\end{enumerate}
\end{proof} 

\subsection{On Frobenius Lie algebras} \label{section: on frobenius lie algebras}

    \indent A Frobenius Lie algebra is just a Lie algebra equipped with an exact symplectic form. It is known that Frobenius Lie algebras are nonunimodular (see \cite[Proposition 3.4]{DM}, although an adaptation of the proof of Proposition \ref{prop: el primer término es positivo} shows the same result).  

    \indent It turns out that the $\t_0$-part inside a DS-contact Lie algebra is a Frobenius Lie algebra with bracket $\alpha$. This generalizes the first half of \cite[Proposition 4.1]{ACN}, where the same result is established for K-contact Lie algebras with $\dim \q = 2$. Recall $\ell \in \g^*$ and $\sigma \in \alt^2 \g^*$ from equations \eqref{eq: ell form} and \eqref{eq: sigma form}, and note that the differentials $d_{\t}$ and $d_{\t_0}$ can be taken as equal upon the identification $\alt_{\xi}^* \t \cong \alt^* \t_0^*$ (see Remark~\ref{obs: when are dg and dh equal}). Denote also by the same name the restrictions of $\ell$ and $\sigma$ to $\t_0$. 
\begin{theorem} \label{thm: t_0 is frobenius}
    $(\t_0, \omega_{\t})$ is a Frobenius Lie algebra. In fact, $\omega_{\t} = d_{\t_0} \ell$. 
\end{theorem}
\begin{proof}
    Proposition \ref{prop: propiedades de ell} (ii) shows $d_{\t} \ell = \omega_{\t} - \sigma$ on $\t$. As already observed, $d_{\t} \ell = d_{\t_0} \ell$ because $\iota_{\xi} \ell = 0$; on the other hand, since $C = 0$, Lemma \ref{lemma: se usa luego} forces $\sigma = 0$.    
\end{proof}
    \indent The other half of \cite[Proposition 4.1]{ACN} is harder to generalize to the DS-contact setting. A reason for that is the presence of the vector $H_0$, which should be replaced by a more complicated object. We note, however, that the $1$-form $\mu$ in Corollary \ref{cor: some obstructions} is similar to the one used in said result. 
\begin{corollary} \label{cor: t_0 is not unimodular}
    Neither $\t_0$ nor $\t$ is unimodular. 
\end{corollary}  
\begin{proof}
    $\t_0$ is not unimodular because it has an exact symplectic form, due to Theorem \ref{thm: t_0 is frobenius}. $\t$ is not unimodular because it is the contactization of $\t$ via $\omega_{\t}$, due to Remark \ref{obs: iff unimod, nil, solv}. 
\end{proof}    
    \indent Theorem \ref{thm: t_0 is frobenius} comes with an associated \textit{realization problem}: which Frobenius Lie algebras can be realized as the $\t_0$-part of some DS-contact Lie algebra? Corollary \ref{cor: some obstructions} can be regarded as a list of obstructions for this to happen, at least when $\b \neq 0$. We explore this question in the next few results. To that aim, we set, for a given Frobenius Lie algebra $(\a, \omega_{\a})$, 
\begin{align*} 
    \mathcal{R}(\a, \omega_{\a}) := \{d \in 2 \mathbb N \mid (\a, \omega_{\a}) \text{ is realizable with $\dim \q = d$} \}. 
\end{align*}
    \noindent We omit the clarification that the realization is ``as the $\t_0$-part of a DS-contact Lie algebra" to~save~space. We start with two positive results.     
\begin{proposition} \label{prop: realizable 1}
    $2\in \mathcal{R}(\a, \omega_{\a})$ for any Frobenius Lie algebra $(\a, \omega_{\a})$. 
\end{proposition} 
\begin{proof}
    Pick $\alpha \in \a^*$ such that $d_{\a} \alpha = \omega_{\a}$, a $2$-dimensional vector space $(\q,\omega_{\q})$, and any invertible $A\in \sp(\q,\omega_{\q})$. On the vector space $\g$ given by $\R\xi \oplus \a \oplus \q$, let
\begin{align*}
    [\xi,\a]_{\g} = 0, \quad  [\xi,u]_{\g} = Au, \quad  [\a,\q]_{\g} = 0, \quad  [x,y]_{\g} =[x,y]_{\a}, \quad [u,v]_{\g} = \omega_{\q}(u,v)\,\xi.
\end{align*}
    \noindent define a skew-symmetric bilinear product on $\g$. A straightforward computation shows that $[\cdot, \cdot]_{\g}$ is indeed a Lie bracket, a fact that can be read off from Lemma \ref{lemma: casi todo jacobi} and Lemma \ref{lemma: casi todo jacobi DS}: we stress the crucial fact that the Jacobi identity on triples $(u,v,w)$ holds trivially because $\dim \q = 2$. Also, let $\theta \in \g^*$ be dual to $\xi$, extend $\alpha$ by $0$ on $\R\xi \oplus \q$, and set $\eta := \theta - \alpha$. Since $-(d_{\g}\eta) \vert_{\a} = \omega_{\a}$ and $-(d_{\g}\eta) \vert_{\q}=\omega_{\q}$, it follows that $\eta$ is contact with Reeb vector $\xi$. Finally, $\ad_{\xi}$ vanishes on $\R\xi \oplus \a$ and equals $A$ on $\q$, hence $\ker \ad_{\xi}=\R\xi \oplus \a$ and $\im \ad_{\xi}=\q$.  
\end{proof} 
    \indent It is not hard to see that $\mathcal{R}(\mathfrak{aff}(\R)) = 2 \N$ for the non-abelian $2$-dimensional Lie algebra $\mathfrak{aff}(\R)$. A more general result can be established with the same effort. 
\begin{proposition}\label{prop: realizable 2}
    $\mathcal{R}(\a, \omega_{\a}) = 2\mathbb N$ for any Frobenius Lie algebra $(\a, \omega_{\a})$ with~a~one-dimensional ideal. 
\end{proposition} 
\begin{proof}
    Let $\R z$ be the one-dimensional ideal in $\a$. Let $\chi \in \a^*$ be defined by the condition that $[x,z]_{\a} = \chi(x) z$ for all $x \in \a$. Choose $\alpha \in \a^*$ such that $d_{\a} \alpha = \omega_{\a}$, which is readily shown to be closed. The nondegeneracy of $\omega_{\a}$ implies that $\alpha(z) \neq 0$, for otherwise every $x \in \a$ would satisfy
\begin{align*}
    \omega_{\a}(x,z) = d_{\a} \alpha(x,z) = - \alpha([x,z]_{\a}) = -\chi(x) \alpha(z).
\end{align*}
    \noindent We may assume $\alpha(z)= -1$, replacing $z$ by a nonzero scalar multiple if necessary. Pick a symplectic vector space $(\q,\omega_{\q})$ of dimension $2m$, and let $A \in \sp(\q,\omega_{\q})$ be invertible. Define
\begin{align*}
    \rho:\a \to \End(\q), \quad \rho(x) := \frac{\chi(x)}{2}\Id + \alpha(x)A.
\end{align*}
    \noindent On the vector space $\g$ given by $\R \xi \oplus \a \oplus \q$, let 
\begin{gather*}
    [\xi,\a]_{\g} = 0, \quad [\xi,u]=Au, \quad [x,u]_{\g} = \rho(x)u, \\
    [x,y]_{\g}=[x,y]_{\a} + \omega_{\a}(x,y) \xi, \quad [u,v]_{\g} = \omega_{\q}(u,v)(\xi+z)
\end{gather*}
    \noindent define a skew-symmetric bilinear product on $\g$; here $x$, $y\in \a$ and $u$, $v\in \q$. We now verify that Jacobi identity holds: 
\begin{itemize}
    \item Since $\R\xi \oplus \a$ is just  the contactization of $(\a,\omega_{\a})$, Jacobi holds for the triples on $\R\xi \oplus \a$.
    \item For $(x,y,u)$ with $x$, $y\in \a$ and $u\in \q$, we have
\begin{align*}
    \rho([x,y]_{\a}) + \omega_{\a}(x,y)A = \frac{\chi([x,y]_{\a})}{2} \Id + \alpha([x,y]_{\a})A + \omega_{\a}(x,y)A = 0 = [\rho(x),\rho(y)]. 
\end{align*}
    \noindent This follows from the following facts: $\chi$ is a character, $d_{\a} \alpha = \omega_{\a}$, $\Id$ and $A$ commute. 
    \item For $(x,u,v)$ with $x\in \a$ and $u$, $v\in \q$, since $[x,z]_{\a} = \chi(x) z$ and $\omega_{\a}(x,z) = \chi(x)$, as well as $\rho(x)^* \omega_{\q} = \chi(x) \omega_{\q}$, we have
\begin{align*}
    [x,[u,v]]_{\g} &= \omega_{\q}(u,v) [x,\xi+z]_{\g} = \omega_{\q}(u,v) \chi(x)(\xi + z) \\
    &= (\rho(x)^* \omega_{\q})(u,v) (\xi + z) = [\rho(x) u,v]_{\g} + [u, \rho(x) v]_{\g}. 
\end{align*} 
    \item Since $A \in \sp(\q,\omega_{\q})$, Jacobi holds for $(\xi,u,v)$ with $u$, $v\in \q$. 
    \item For $(u,v,w)$ with $u$, $v$, $w\in \q$, it suffices to note that
\begin{align*}
    [u,\xi+z]_{\g} = -Au-\rho(z)u = -Au + Au = 0, 
\end{align*}
    \noindent stemming from the facts that $\chi(z) = 0$ and $\alpha(z) = - 1$. 
\end{itemize}     
    \indent Let $\eta\in \g^*$ be dual to $\xi$. Since $[\xi,\a] = 0$, $[\xi,\q] \subseteq \q$, and $\eta \vert_{\q}=0$, it follows that $d_{\g}\eta(\xi,\cdot)=0$. Also, $-(d_{\g}\eta)\vert_{\a} = \omega_{\a}$ and $-(d_{\g}\eta)\vert_{\q} = \omega_{\q}$, while the mixed terms vanish because $[\a,\q] \subseteq \q$. Hence $-(d_{\g}\eta) \vert_{\ker\eta} = \omega_{\a} + \omega_{\q}$ is symplectic, so $\eta$ is contact with Reeb vector $\xi$. By construction, $\ad_{\xi}$ vanishes on $\R\xi\oplus \a$ and equals $A$ on $\q$, hence $\ker \ad_{\xi}=\R\xi\oplus \a$ and $\im \ad_{\xi}=\q$. 
\end{proof}
    \indent The combination of Corollary \ref{cor: beta neq when dim q geq 4} and Corollary \ref{cor: some obstructions} (i) is enough to rule out some Frobenius Lie algebras from being realizable for specific dimensions of $\q$. To see this, let $(\a, \omega_{\a})$ be a Frobenius Lie algebra and set
\begin{align*}
    \nu(\a) := \min\{ \dim J \mid \text{$J$ is a nonzero ideal in $\a$} \}.
\end{align*}    
\begin{corollary} \label{cor: obstruction}
    If $k^2 < \nu(\a)$ for some $k \geq 2$ then $2k \notin \mathcal{R}(\a, \omega_{\a})$. 
\end{corollary}
\begin{proof}
    Every realization of a given Frobenius Lie algebra $(\a, \omega_{\a})$ with $\q$ of dimension $\dim \q = 2k$ with $k \geq 2$ forces $\a$ to have a nontrivial ideal of dimension less than $k^2$, namely $\b = \im \beta$. 
\end{proof}
    \indent We illustrate Corollary \ref{cor: obstruction} with an example. Consider the Lie algebra $\mathfrak{gl}(n,\R) \ltimes V$, where $V = \R^n$, where the Lie bracket $[(X,u), (Y,v)]$ is given by $([X,Y], Xv - Yu)$. Notice that $V$ is the unique minimal nonzero ideal of $\a_n$, which has dimension $n$. Indeed, first notice that it cannot contain proper ideals since it is irreducible as a $\mathfrak{gl}(n,\R)$-submodule. Now pick $0 \neq J \lhd \a_n$, and argue that $V \subseteq J$ by considering two different cases: 
 \begin{itemize}
     \item If $J \cap V \neq 0$ then $J \cap V$ is a nonzero $\mathfrak{gl}(n,\R)$-submodule of the standard module $V$, which is irreducible, and so $J \cap V = V$.
     \item If $J \cap V = 0$ then we arrive at a contradiction. For if $(A,u) \in I$ is nonzero then $A \neq 0$, for otherwise $U \neq 0$ and so $(0,u) \in J \cap V$; thus there is some $w \in V$ for which $Aw \neq 0$, and then $[(A,u),(0,w)] = (0,Aw) \in J \cap V$. 
 \end{itemize}
    \noindent Let $T := \diag(1,\dots,n)$ and let $b\in V^*$ be the row vector $(1, \dots, 1)$. Define $\lambda \in \a_n^*$ by $\lambda(X,u) := \tr(TX) + b(u)$, and set $\omega_{\a_n} := d \lambda$. We claim that $\omega_{\a_n}$ is nondegenerate, and thus an exact symplectic form on $\a_n$. To see this, take  $(X,u) \in \a_n$ in the kernel of $\omega_{\a_n}$. Then, for all $(Y,v) \in \a_n$,
\begin{align*}
    0 &= - \omega_{\a_n} ((X,u), (Y,v) ) = \lambda ([(X,u),(Y,v)] ) \\
    &= \tr(T[X,Y]) + b(Xv-Yu) = \tr (([T,X] - ub)Y) + (bX)(v).
\end{align*}
    \noindent Hence, $bX = 0$ and $[T,X] = ub$. The diagonal of $[T,X]$ is zero,  whereas the diagonal of $ub$ is just the vector $u$. Thus $u = 0$, and therefore $[T,X] = 0$. Since $X$ commutes with $T$, $X$ must be diagonal. The fact that $bX = 0$ means that every column sum of $X$ is zero, and so $X = 0$. To sum up, the kernel of $\omega_{\a_n}$ is trivial. We have thus shown the following: 
\begin{proposition} \label{prop: realizable 3}
    For the Frobenius Lie algebra $(\a_n, \omega_{\a_n})$ described above, $2k \notin \mathcal{R}(\a_n, \omega_{\a_n})$ for any $k \in \N$ such that $k^2 < n$. 
\end{proposition} 

    \indent It is unclear at the moment whether $\mathcal{R}(\a, \omega_{\a})$ actually depends on (the equivalence class of) the exact symplectic form $\omega_{\a}$. Note that both Proposition \ref{prop: realizable 1} and Proposition \ref{prop: realizable 2} are valid irrespective of the chosen symplectic form $\omega_{\a}$, so there is some evidence against it. A related observation is as follows: Fix a realization $\g = \R \xi \oplus \a \oplus \q$ of a Frobenius Lie algebra $(\a, \omega_{\a})$ with contact form $\eta$, which comes with their associated tensor $\beta: \alt^2 \q^* \to \a$ and some symplectic form $\omega_{\q}$ on $\q$. Let $\sigma_{\a}$ be another exact symplectic form on $\a$, and choose $\tau \in \a^*$ such that $d \tau = \sigma_{\a} - \omega_{\a}$. We can extend $\tau$ by $0$ on $\R \xi \oplus \q$, and define $\eta_{\tau} := \eta - \tau$. Notice that
\begin{align*}
    - (d \eta_{\tau})\vert_{\a} = \sigma_{\a}, \quad - (d \eta_{\tau})\vert_{\q} = \omega_{\q} - \tau \circ \beta,
\end{align*}
    \noindent so $\eta_{\tau}$ is again a contact form on $\g$ with the same Reeb vector if and only if $\omega_{\q} - \tau \circ \beta$ is symplectic on $\q$. When this happens, $(\a, \sigma_{\a})$ is realizable on the same space as $(\a, \omega_{\a})$, the only two differences being the contact form $\eta_{\tau}$ on $\g$ and the symplectic form $\omega_{\q} - \tau \circ \beta$ on $\q$. In particular, $\mathcal{R}(\a, \omega_{\a}) = \mathcal{R}(\a, \sigma_{\a})$. This situation is, however, rather special, so little intuition comes out of it. 

\subsection{A classification result in dimension five} \label{section: a classification result in dimension five}
    \indent As remarked below Proposition \ref{prop: contactization}, contact Lie algebras with nontrivial center are in bijective correspondence with symplectic Lie algebras, via $1$-dimensional central extension by a symplectic cocycle. Thus, in view of a classification of DS-contact Lie algebras, it is natural to split the problem according to whether the center is trivial or not. In dimension $5$, the case of nontrivial center reduces to the classification of $4$-dimensional symplectic Lie algebras, which is known by \cite{Ovando}. Hence only the centerless case is relevant here. 
\begin{proposition} \label{prop: 5-dimensional DS-contact Lie algebras}
    A centerless $5$-dimensional DS-contact Lie algebra is isomorphic to exactly one of the following Lie algebras:  
\begin{enumerate} [\rm (i)]
    \item The Lie algebra $\g_{0,+}$ with basis $\{\xi,x,y,u,v\}$ and nonzero brackets given by
    \begin{gather*}
        [x,y] = \xi+y, \quad [u,v] = \xi, \quad [\xi,u] = v, \quad [\xi,v] = -u, \\ 
        [y,u] = -v, \quad [y,v] = u.
    \end{gather*} 
    \item The Lie algebra $\g_{1,+}$ with basis $\{\xi,x,y,u,v\}$ and nonzero brackets given by
    \begin{gather*}
        [x,y] = \xi+y, \quad [u,v] = \xi, \quad [\xi,u] = u, \quad [\xi,v] = -v, \\
        [y,u] = -u, \quad [y,v] = v.
    \end{gather*} 
    \item The Lie algebra $\g_{0,-}$ with basis $\{\xi,x,y,u,v\}$ and nonzero brackets given by
    \begin{gather*}
        [x,y] = \xi+y, \quad [u,v] = \xi+y, \quad [\xi,u] = v, \quad [\xi,v] = -u, \\
        [x,u] = \tfrac{1}{2} u, \quad [x,v] = \tfrac{1}{2} v, \quad [y,u] = -v, \quad [y,v] = u.
    \end{gather*}
    \item The Lie algebra $\g_{1,-}$ with basis $\{\xi,x,y,u,v\}$ and nonzero brackets given by
    \begin{align*}
        [x,y] = \xi+y, \quad [u,v] = \xi+y, \quad [\xi,u] = u, \quad [\xi,v] = -v, \\
        [x,u] = \tfrac{1}{2} u, \quad [x,v] = \tfrac{1}{2} v, \quad [y,u] = -u, \quad [y,v] = v.
    \end{align*}
\end{enumerate} 
\end{proposition}
\begin{proof}
    One has $\dim \t = 3$. Indeed, $\dim \t = 5$ would force $\g$ to have nontrivial center, while $\dim \t = 1$ would make $\g$ $3$-dimensional by Proposition \ref{prop: t = R xi means sl(2;R) or su(2)}. Equivalently, $\dim \t_0 = 2$. Since $\t_0$ is not abelian, by Corollary \ref{cor: t_0 is not unimodular} it follows that $\t_0 \cong \aff(\R)$. Thus there is a basis $\{x,y\}$ of $\t_0$ such that $[x,y]_{\t_0} = y$ and $\omega_{\t}(x,y) = 1$, and therefore
\begin{align*}
    [x,y]_{\t} = \omega_{\t}(x,y)\xi + [x,y]_{\t_0} = \xi + y,
\end{align*}
    \noindent as $\t$ is the contactization of $\t_0$ by the cocycle $\omega_{\t}$. Since $\dim \q = 2$, we may also choose a basis $\{u,v\}$ of $\q$ such that $\omega_{\q}(u,v) = 1$. Then
\begin{align*}
    [u,v] = \omega_{\q}(u,v)\xi + \beta(u,v) + \gamma(u,v) = \xi + \beta + \gamma,
\end{align*} 
    \noindent where $\beta \in \t_0$ and $\gamma \in \q$ collapse to just two vectors since $\alt^2 \q^*$ is 1-dimensional. From Corollary \ref{cor: gamma = 0 when dim q = 2} we get $\gamma = 0$. Next, Lemma \ref{lemma: casi todo jacobi DS} (ii) ensures that both $\rho(x)$ and $\rho(y)$ lie in the centralizer of $A$ in $\End(\q)$; since $A$ is an invertible trace-free endomorphism of the $2$-dimensional symplectic space $(\q,\omega_{\q})$, its centralizer in $\End(\q)$ is $\Span\{ \Id, A \}$. Hence, there are $a$, $b$, $c$, $d \in \R$ for which
\begin{align*}
    \rho(x) = a \Id + bA, \quad \rho(y) = c \Id + dA.
\end{align*}
    \noindent In particular, $\rho(x)$ and $\rho(y)$ commute. Hence, by Lemma \ref{lemma: casi todo jacobi} (vii),
\begin{align*}
    0 = [\rho(x), \rho(y)] = \rho(\alpha(x,y)) + \omega_{\t}(x,y) A = \rho(y) + A,
\end{align*}
    \noindent and so $\rho(y) = - A$. Now write $\beta = sx+ty \in \t_0$ for some $s$, $t \in \R$. Lemma \ref{lemma: casi todo jacobi} (viii) yields
\begin{gather*}
    2a = \tr(\rho(x)) = (\rho(x)^* \omega_{\q})(u,v) = \omega_{\t}(x,\beta)) = t, \\
    0 = - \tr(A) = \tr(\rho(y)) = (\rho(y)^* \omega_{\q})(u,v) = \omega_{\t}(y,\beta)) = - s,
\end{gather*}
   \noindent hence $\beta = 2a y$. Then, Lemma \ref{lemma: casi todo jacobi} (ix) yields 
\begin{align*}
    4a^2 y = 2a (2a y) = \tr(\rho(x)) \beta = (\rho(x)^* \beta)(u,v) = \alpha(x, \beta) = 2a \alpha(x,y) = 2ay,
\end{align*}  
    \noindent from where $0 = 2a (1 - 2a) y$, and therefore either $a = 0$ or $a = \frac{1}{2}$. Thus, since replacing $x \mapsto x - b \xi$ preserves the relation $[x,y] = \xi + y$ and changes $\rho(x)$ to $\rho(x) - b A$, we may assume  
\begin{align} \label{eq: choice of rho and beta}
    \text{either both $\rho(x)$ and $\beta$ are $0$, or $\rho(x) = \frac{1}{2} \Id$ and $\beta = y$}. 
\end{align}
    \noindent The fact that $A \in \sp(\q, \omega_{\q})$ is invertible on a $2$-dimensional symplectic space means that the basis $\{u,v\}$ can always be chosen so that
\begin{align} \label{eq: choice of A}
    \text{either $A = c J$ or $A = c H$ for some $c \neq 0$, where $J :=  \begin{pmatrix} 0 & -1 \\ 1 & \phantom{+} 0\end{pmatrix}$ and $H := \begin{pmatrix} 1 & \phantom{+} 0 \\ 0 & - 1\end{pmatrix}$}. 
\end{align}
    \noindent The discussion now splits into several cases. First, if $\beta = 0$, then we can set $z := \xi + y$ to obtain
\begin{align*}
    \g := \Span\{x,z\}\oplus\Span\{\xi,u,v\};
\end{align*}
    \noindent the first factor is $\aff(\R)$, while the second is a $3$-dimensional simple Lie algebra determined by 
\begin{align*}
    [u,v]=\xi, \quad \ad_\xi|_{\Span\{u,v\}}=A.
\end{align*}
    \noindent If $A = cJ$ with $c>0$, this factor is $\su(2)$; if $A=cJ$ with $c<0$, it is $\sl(2,\R)$; if $A=cH$, it is again $\sl(2,\R)$. Thus the case $\beta=0$ gives precisely $\g_{0,+}$ and $\g_{1,+}$. Assume now that $\beta=y$, and set $z:=\xi+y$ once again. Then
\begin{align*}
    [x,z] = z, \quad [u,v] = z, \quad
    [x,u] = \tfrac{1}{2}u, \quad [x,v] = \tfrac{1}{2}v.
\end{align*}
    \noindent In this case, the rescaling $\xi \mapsto \xi' := c^{-1}\xi$ and $y \mapsto y' := z - \xi'$ normalizes $A$ to either $J$ or $H$, while preserving the brackets $[x,y']=\xi'+y'$ and $[u,v]=\xi'+y'$. >Hence, the case $\beta=y$ gives precisely $\g_{0,-}$ and $\g_{1,-}$. 
\end{proof}
    \indent For $\g_{0,+}$ and $\g_{1,+}$, set $z := \xi + y$. Since $[x,z] = z$ and $z$ commutes with $\xi$, $u$, and $v$, one readily checks that $\g$ splits as a direct sum of ideals $\Span\{x,z\}$ and $\Span\{\xi,u,v\}$. Moreover, as discussed during the proof, 
\begin{align*}
    \g_{0,+} \cong \aff(\R) \oplus \su(2), \quad \g_{1,+} \cong \aff(\R) \oplus \sl(2,\R).  
\end{align*}
    \noindent Likewise, for $\g_{0,-}$, the basis
\begin{align*}
    e_1 := x, \quad e_2 := \xi, \quad e_3 := -(\xi+y), \quad e_4 := u, \quad e_5 := v
\end{align*}
    \noindent yields the nonzero brackets
\begin{gather*}
    [e_1,e_3] = e_3, \quad [e_1,e_4] = \tfrac{1}{2} e_4, \quad [e_1,e_5] = \tfrac{1}{2} e_5, \\
    [e_2,e_4] = e_5, \quad [e_2,e_5] = -e_4, \quad [e_4,e_5] = -e_3,
\end{gather*}
    \noindent so $\g_{0,-}$ is precisely the Lie algebra denoted by $\g_0$ in both \cite[Theorem 13]{AFV} and \cite[Theorem 4.7]{CF}. It follows from those classifications that $\g_{0,+}$, $\g_{1,+}$, and $\g_{0,-}$ admit compatible Sasakian structures, whereas $\g_{1,-}$ does not. Thus $\g_{1,-}$ is the only strict example in the list (i.e., not admitting K-compact compatible metrics). 

    \indent Notice that both $\g_{0,-}$ and $\g_{1,-}$ are solvable (indeed, the commutator ideal for both of them is $\Span\{\xi + y\}$), while neither $\g_{0,+}$ nor $\g_{1,+}$ are solvable (their derived series stabilizes at $\Span \{\xi,u,v\}$). Both $\g_{0,-}$ and $\g_{1,-}$ appear on the list given in \cite[Section 4.2]{D}, being isomorphic to the Lie algebras 22 and 21, respectively. 
    
    \indent It seems plausible that similar arguments could be used to classify DS-contact Lie algebras in higher dimensions, provided classifications of lower-dimensional exact and nonexact symplectic Lie algebras are available. It is less clear whether such a classification would be worth pursuing, since the computations appear to grow quickly out of hand.

\subsection{More cohomology remarks} \label{section: more cohomology remarks}

    \noindent For the $1$-form $\ell \in \t^*$ defined as the restriction of the form defined in equation \eqref{eq: ell form}, which satisfies $\omega_{\t} = d_{\t_0} \ell$ due to Theorem \ref{thm: t_0 is frobenius}, define 
\begin{align*}
    \delta_{\ell} := \omega_{\q} + \ell \circ \beta \in \alt^2 \q^*.
\end{align*}
    \noindent Denote by $\beta^*: Z^1_{CE}(\t_0) \to \alt^2 * \q^*$ the map given by $\beta^*([\sigma]) = \sigma \circ \beta$. Define
\begin{align*}
    F_{\ell}:\R \to Z^1_{CE}(\t_0) \to \alt^2 \q^*, \quad F_{\ell}(c, [\sigma]) = c \delta_{\ell} + \sigma \circ \beta. 
\end{align*}
    \noindent The map $F_{\ell}$ is merely a contraption to detect closed forms on $\g$, via $\ker F_{\ell} = Z^1_{CE}(\g)$. To see this set $\varphi \in \t^*$ to be  $\varphi = \eta + \ell$, which is closed on $\t$ due to Lemma \ref{lemma: simplecticidad} (v) and Theorem \ref{thm: t_0 is frobenius}. 
\begin{lemma} \label{lemma: ker F ell}
    A $1$-form $\lambda\in\g^*$ is closed if and only if $\lambda \vert_{\q} = 0$ and $\lambda \vert_{\t} = c \varphi + \sigma$ for some $c \in \R$ and $\sigma \in Z^1_{CE}(\t_0)$ satisfying $F_{\ell}(c,[\sigma])=0$.
\end{lemma}    
\begin{proof}
    Every closed $1$-form vanishes on $\q = \im(\ad_{\xi}) \subseteq [\g,\g]$. Write $\lambda \vert_{\t} = c \eta \vert_{\t} + \tau$ for some $\tau \in \t_0^*$, and note that
\begin{align*}
    d_{\t}(\lambda\vert_{\t}) = - c \omega_{\t} + d_{\t_0} \tau = c d_{\t_0} (\ell + \tau)   
\end{align*}
    \noindent Thus $\lambda \vert_{\t}$ is closed on $\t$ if and only if there is $\sigma \in Z^1_{CE}(\t_0)$ for which $\tau = c \ell + \sigma$. The remaining conditions on $c$ and $\sigma$ come from the fact that $[\t,\q] \subseteq \q$, since for all $u$, $v \in \q$ we have
\begin{align*}
    0 &= \lambda([u,v]) =c \omega_{\q}(u,v) + c \ell(\beta(u,v)) + (\sigma \circ \beta)(u,v). \qedhere 
\end{align*}
\end{proof}
    \indent The next is a refined version of Lemma \ref{lemma: xi en el conmutador}. Set $b_1(\g) := \dim H^1_{CE}(\g)$
\begin{proposition}\label{prop: xi in commutator}
    The following conditions are equivalent on any DS-contact Lie algebra:
\begin{multicols}{3}
\begin{enumerate}[\rm (i)]
    \item $\xi \in [\g,\g]$.  
    \item $\delta_{\ell} \notin \im(\beta^*)$.
    \item $b_1(\g) = \dim\ker(\beta^*)$.
\end{enumerate}    
\end{multicols}
\end{proposition}
\begin{proof}  \phantom{.} \\
    \noindent (i) $\Rightarrow$ (ii): If there exists $\sigma\in Z^1_{CE}(\t_0)$ such that $\beta^*[\sigma] = -\delta_{\ell}$ then $(1,[\sigma]) \in \ker F_{\ell}$, which by Lemma \ref{lemma: ker F ell} gives a closed $1$-form vanishing whose value on $\xi$ is $1$, contradicting $\xi\in[\g,\g]$. 

    \noindent (ii) $\Rightarrow$ (iii): If $(c,[\sigma]) \in \ker F_{\ell}$ then $c \delta_{\ell} \in \im(\beta^*)$. By (ii), this forces $c=0$. Thus $\ker F_{\ell} = \ker(\beta^*)$, and so $b_1(\g) = \dim\ker(\beta^*)$.

    \noindent (iii) $\Rightarrow$ (i): If $\xi\notin[\g,\g]$, there exists a closed $1$-form $\lambda$ with $\lambda(\xi) = 1$. By Lemma \ref{lemma: ker F ell}, this gives an element $(1,[\sigma_0])\in\ker F_{\ell}$, and hence we get the identification
\begin{align*}
    \ker F_{\ell} = \ker(\beta^*)\oplus \R(1,[\sigma_0]),
\end{align*}
    \noindent implying $b_1(\g)=\dim\ker(\beta^*)+1$. \qedhere 
\end{proof}
    \indent The fact that $\omega_{\t}$ is an exact symplectic form on the Lie algebra $(\t_0, \alpha)$ entails that $(\t, \eta \vert_{\t})$ is a cohomologically trivial 1-dimensional central extension of $(\t_0, \alpha, \omega_{\t})$. While the exactness of $\omega_{\t}$ in $(\t_0, \alpha)$ forces $\xi \notin [\t,\t]_{\t}$ (this is a consequence of the universal coefficient theorem, as pointed out in \cite[Remark 3.11]{AG3}), it is often not true that $\xi \notin [\g,\g]_{\g}$: indeed, the Lie algebras in both Example \ref{example: sasakian 2} and Example \ref{example: sasakian 2} are DS-contact (in fact, they admit compatible Sasakian structures) and satisfy $\xi \in [\g,\g]_{\g}$.

\


\begin{thebibliography}{99}

\bibitem{AHK}
D.\ Alekseevsky, K.\ Hasegawa, and Y. \ Kamishima. Homogeneous Sasaki and Vaisman manifolds of unimodular Lie groups. \textit{Nagoya Math. J.} \textbf{243}, 83--96 (2021).

\bibitem{ARVS 1}
M.\ A.\ Alvarez, M.\ C.\ Rodríguez-Vallarte, and G.\ Salgado. Contact and Frobenius solvable Lie algebras with abelian nilradical. \textit{Commun. Algebra} \textbf{46}, No. \textbf{10}, 4344-4354 (2018).

\bibitem{ARVS 2}
M.\ A.\ Alvarez, M.\ C.\ Rodríguez-Vallarte, and G.\ Salgado. Contact nilpotent Lie algebras. \textit{Proc. Am. Math. Soc.} \textbf{145}, No. \textbf{4}, 1467-1474 (2017).

\bibitem{ARVS 3}
M.\ A.\ Alvarez, M.\ C.\ Rodríguez-Vallarte, and G.\ Salgado. Deformation theory of contact Lie algebras as double extensions. 
\textit{Proc. Am. Math. Soc.} \textbf{149}, No. \textbf{5}, 1827-1836 (2021).

\bibitem{AFV} 
A.\ Andrada, A.\ Fino, L.\ Vezzoni. A class of Sasakian 5-manifolds. \textit{Transform. Groups} \textbf{14}, No. \textbf{3}, 493--512 (2009). 

\bibitem{ACN} 
A.\ Andrada, S.\ G.\ Chiossi, and A.\ J.\ Nuñez. $\eta$-Einstein Sasakian Lie algebras. \textit{Manuscripta Math}. \textbf{177}, No. \textbf{1}, article 12 (2026). 

\bibitem{AG1}
A.\ Andrada and A.\ Garrone. Construction of symplectic solvmanifolds satisfying the hard-Lefschetz condition. \textit{Linear Algebra Appl.} \textbf{706}, 70--100 (2025). 

\bibitem{AG2}
A.\ Andrada and A.\ Garrone. Symplectic solvmanifolds not satisfying the hard-Lefschetz condition. Preprint, arXiv:2505.08113 [math.DG] (2025).

\bibitem{AG3}
A.\ Andrada and A.\ Garrone. 1-Lefschetz contact solvmanifolds. Preprint, arXiv: 2512.24311 [math.DG] (2025).

\bibitem{Blair}
D.\ E.\ Blair. \textit{Riemannian Geometry of Contact and Symplectic Manifolds}. Progress in Mathematics \textbf{203}, Birkhäuser (2010).

\bibitem{BEM} 
M.\ S.\ Borman, Y.\ Eliashberg, and E.\ Murphy. Existence and classification of overtwisted contact structures in all dimensions.
\textit{Acta Math}. \textbf{215}, No. \textbf{2}, 281-361 (2015).

\bibitem{BW}
W.\ M.\ Boothby and H.\ C.\ Wang. On contact manifolds. \textit{Ann. Math.} \textbf{68}, 721--734 (1958).

\bibitem{Cagliari} 
B.\ Cappelletti-Montano, A.\ de Nicola, and I.\ Yudin. Hard Lefschetz theorem for Sasakian manifolds. \textit{J. Differ. Geom}. \textbf{101}, No. \textbf{1}, 47--66 (2015).

\bibitem{CF}
G.\ Calvaruso and A.\ Fino. Five-dimensional K-contact Lie algebras. \textit{Monatsh. Math.} \textbf{167}, No. \textbf{1}, 35--59 (2012).

\bibitem{CR}
V.\ E.\ Coll and N.\ Russoniello. Classification of contact seaweeds. \textit{J. Algebra} \textbf{659}, 811-817 (2024).

\bibitem{CMRS}
V.\ E.\ Coll, N.\ Mayers, N.\ Russoniello, and G.\ Salgado. Contact seaweeds. \textit{Pac. J. Math.} \textbf{320}, No. \textbf{1}, 45-60 (2022).
 
\bibitem{Chu}
B.\ Chu. Symplectic homogeneous spaces. \textit{Trans. Am. Math. Soc.} \textbf{197}, 145--159 (1974). 


\bibitem{D}
A.\ Diatta. Left invariant contact structures on Lie groups. \textit{Differ. Geom. Appl.} \textbf{26}, No. \textbf{5}, 544--552 (2008).

\bibitem{DM}
A.\ Diatta and B.\ Manga. On properties of principal elements of Frobenius Lie algebras. \textit{J. Lie Theory} \textbf{24}, No. \textbf{3}, 849--864 (2014). 

\bibitem{DF}
A.\ Diatta and B.\ Foreman. Lattices in contact Lie groups and 5-dimensional contact solvmanifolds. \textit{Kodai Math. J.} \textbf{38}, No. \textbf{1}, 228--248 (2015).

\bibitem{DPS}
G.\ Dileo, D.\ Poyraz, and B.\ Şahin. Transversely Kähler almost contact metric Lie algebras. Preprint, arXiv:2604.12538 (2026)

\bibitem{KA}
A.\ El Kacimi-Alaoui. Opérateurs transversalement elliptiques sur un feuilletage riemannien et applications. \textit{Compositio Math}. \textbf{73}, No. \textbf{1}, 57--106 (1990). 

\bibitem{FH}
W.\ Fulton and J.\ Harris. \textit{Representation Theory: A First Course}. Graduate Texts in Mathematics \textbf{129}. Springer-Verlag (1991). 

\bibitem{GR}
M.\ Goze and E.\ Remm. Contact and Frobeniusian forms on Lie groups. \textit{Differ. Geom. Appl}. \textbf{35}, 74-94 (2014).

\bibitem{G}
M.\ L.\ Gromov. Stable mappings of foliations into manifolds (Russian). \textit{Izv. Akad. Nauk SSSR Ser. Mat.} \textbf{33}, 707–734 (1969). 

\bibitem{Hano}
J.\ Hano. On Kählerian homogeneous spaces of unimodular Lie groups. \textit{Am. J. Math.} \textbf{79}, 885--900 (1957).

\bibitem{Lin}
Y.\ Lin. Hodge theory on transversely symplectic foliations. \textit{Q. J. Math.} \textbf{69}, No. 2, 585--609 (2018).

\bibitem{LL}
E.\ Loiudice and A.\ Lotta, A. On five dimensional Sasakian Lie algebras with trivial center. \textit{Osaka J. Math.} \textbf{55}, No. \textbf{1}, 39-49 (2018).

\bibitem{MY}
A.\ Moreau and O.\ Yakimova. Coadjoint orbits of reductive type of parabolic and seaweed Lie subalgebras. \textit{Int. Math. Res. Not}. \textbf{2012}, No. \textbf{19}, 4475-4519 (2012).

\bibitem{Nakajima}
K.\ Nakajima. Homogeneous Kähler manifolds of non-degenerate Ricci curvature. \textit{J. Math. Soc. Japan} \textbf{42}, 475--494 (1990).

\bibitem{Ovando}
G.\ Ovando. Four dimensional symplectic Lie algebras. \textit{Beitr. Algebra Geom.} \textbf{47}, No. \textbf{2}, 419--434 (2006).

\bibitem{RVSSV}
M.\ C.\ Rodríguez-Vallarte, G.\  Salgado, and O.\ A.\ Sánchez-Valenzuela. On extensions of Frobenius-Kähler and Sasakian Lie algebras. Preprint, arXiv:2408.11236 (2024).

\bibitem{S}
G.\ Salgado. Invariants of contact Lie algebras. \textit{J. Geom. Phys}. \textbf{144}, 388-396 (2019).

\end{thebibliography}
\end{document}